\documentclass[11pt]{amsart}

\usepackage{enumerate, amsmath, amsthm, amsfonts, amssymb,  mathrsfs, graphicx, lscape}
\usepackage[all]{xy}
\usepackage[usenames, dvipsnames]{color}
\usepackage[margin=1in,marginparwidth=0.8in, marginparsep=0.1in]{geometry} 
\usepackage{tikz}
\usetikzlibrary{matrix,arrows}

\usepackage{pgflibraryarrows}
\usepackage{pgflibrarysnakes}
\usetikzlibrary{fit}

\usepackage{verbatim}

\usepackage{relsize}
\usepackage{tikz-cd}
\newtheorem{theorem}{Theorem}[section]

\newtheorem{proposition}[theorem]{Proposition}
\newtheorem{lemma}[theorem]{Lemma}
\newtheorem{corollary}[theorem]{Corollary}
\newtheorem{conjecture}[theorem]{Conjecture}
\newtheorem*{conjecturenonum}{Conjecture}

\newtheorem{notation}[theorem]{Notation}
\newtheorem{question}[theorem]{Question}
\newtheorem{assumption}[theorem]{Assumption}

\theoremstyle{definition}
\newtheorem{remark}[theorem]{Remark}
\newtheorem{example}[theorem]{Example}
\newtheorem{definition}[theorem]{Definition}
\numberwithin{equation}{section}

\newcommand{\sayR}[1]{\say[R]{#1}}
\newcommand{\sayT}[1]{\say[T]{#1}}

\newcommand{\vs}{\vspace{0.3cm}}
\newcommand{\udim}{\underline{\dim}\,}
\newcommand{\att}{\mathrm{Att}}
\newcommand{\scss}{\mathrm{scss}\,}
\newcommand{\ext}{\mathrm{ext}}    
\newcommand{\mo}{M^{\Theta-\mathrm{st}}_\alpha(Q)}
\newcommand{\ros}{R^{\Theta-\mathrm{st}}_\alpha(Q)}
\newcommand{\Gl}{\mathrm{Gl}}
\newcommand{\scl}{\mathrm{sl}}

\newcommand{\tQ}{\tilde{Q}}

\newcommand{\ses}[3]{0\rightarrow #1\rightarrow #2\rightarrow#3\rightarrow 0}
\newcommand{\sesB}[5]{0\rightarrow #1\xrightarrow{#4} #2\xrightarrow{#5} #3\rightarrow 0}
\newcommand{\sk}[2]{\langle #1,#2\rangle}

\newcommand{\kk}{\Bbbk}
\newcommand{\hn}{\mathrm{hn}}

\newcommand{\Z}{\mathbb{Z}}

\newcommand{\id}{\mathrm{id}}
\newcommand{\Pn}{\mathbb{P}}

\newcommand{\Q}{\mathbb{Q}}
\newcommand{\F}{\mathbb{F}}
\newcommand{\C}{\mathbb{C}}
\DeclareMathOperator{\dimv}{\underline{dim}}

\newcommand{\cB}{\mathcal{B}}

\newcommand{\cE}{\mathcal{E}}

\newcommand{\cR}{\mathcal{R}}

\newcommand{\cU}{\mathcal{U}}

\newcommand{\zl}{\ensuremath{\lambda}}

\newcommand{\ZZ}{\mathbb{Z}}
\newcommand{\CC}{\mathbb{C}}

\newcommand{\T}{\mathbb{T}}
\renewcommand{\phi}{\varphi}
\renewcommand{\emptyset}{\varnothing}

\renewcommand{\tilde}[1]{\widetilde{#1}}

\newcommand{\setst}[2]{\left\{ #1 \mid #2 \right\}}

\makeatletter
\def\Ddots{\mathinner{\mkern1mu\raise\p@
\vbox{\kern7\p@\hbox{.}}\mkern2mu
\raise4\p@\hbox{.}\mkern2mu\raise7\p@\hbox{.}\mkern1mu}}
\makeatother

\DeclareMathOperator{\im}{image} 

\DeclareMathOperator{\coker}{coker}
\newcommand{\Hom}{\operatorname{Hom}}

\DeclareMathOperator{\diag}{diag}
\DeclareMathOperator{\rank}{rank}
\DeclareMathOperator{\Ext}{Ext}

\DeclareMathOperator{\End}{End}

\DeclareMathOperator{\rep}{rep}

\newcommand{\Gr}{\mathbf{Gr}}

\newcommand{\onto}{\twoheadrightarrow}
\newcommand{\into}{\hookrightarrow}
\newcommand{\xto}[1]{\xrightarrow{#1}}

\begin{document}
\title{Tree normal forms for quiver representations}

\author{Ryan Kinser}
\address{Department of Mathematics, University of Iowa, Iowa City, Iowa 52242, USA}
\email{ryan-kinser@uiowa.edu}

\author{Thorsten Weist}
\address{Bergische Universit\"at Wuppertal, Gau{\ss}str. 20, 42097 Wuppertal, Germany}
\email{weist@uni-wuppertal.de}

\begin{abstract}
We explore methods for constructing normal forms of indecomposable quiver representations. The first part of the paper develops homological tools for recursively constructing families of indecomposable representations from indecomposables of smaller dimension vector. This is then specialized to the situation of tree modules, where the existence of a special basis simplifies computations and gives nicer normal forms. Motivated by a conjecture of Kac, we use this to construct cells of indecomposable representations as deformations of tree modules. The second part of the paper develops geometric tools for constructing cells of indecomposable representations from torus actions on moduli spaces of representations. As an application we combine these methods and construct families of indecomposables - grouped into affine spaces - which actually gives a normal form for all indecomposables of certain roots.
\end{abstract}



\maketitle

\setcounter{tocdepth}{1}
\tableofcontents

\section{Introduction}
\subsection{Background and Motivation}\label{intro1}
A central problem in the theory of finite-dimensional algebras is not only to determine all indecomposable representations of an algebra, but also to give normal forms, grouped into meaningful families when possible.  Of course one does not hope to accomplish this uniformly for all algebras, but rather to develop techniques that can be applied to certain classes of quivers and dimension vectors.  

This article contributes to this program by constructing families of indecomposable representations which can be thought of as deformations of a given quiver representation $M$, under suitable conditions. Optimally, these deformations are given by an affine space which, moreover, parametrizes pairwise non-isomorphic indecomposable representations. If $M$ has a nice structure, e.g. if it is a tree module, we immediately get a normal form for the deformed representations.

A main theme of our work is that we should not expect a single most general method to construct indecomposables, but many techniques with incomparable assumptions which can be used in parallel. 
Thus the methods and results of this paper come in two distinct flavors: homological and geometric, which can be combined to construct different kind of indecomposable representations. For many dimension vectors, this gives a partial classification of the indecomposable representations including a normal form. For certain dimension vectors, we even obtain a full classification. In all these cases, we can show that this subset of the set of all isomorphism classes of indecomposables has a cellular decomposition into affine spaces.

These kind of decompositions into affine spaces are particularly interesting as our methods aim to bring some understanding to a conjecture of V. Kac from the early 1980s.
Fixing a quiver $Q$ and a dimension vector $\alpha$, define $a_\alpha(q)$ as the number of absolutely indecomposable representations over $\mathbb F_q$ of dimension $\alpha$ (i.e. those which remain indecomposable after extension of scalars to an algebraic closure of $\mathbb F_q$).
Kac proved that the function $a_\alpha(q)$ is polynomial in $q$ with integer coefficients, $a_\alpha(q)=\sum_{i=0}^nc_iq^{i}$ for some $c_i \in \mathbb{Z}$, and conjectured \cite[Conjecture 2]{Ka83} that each $c_i \geq 0$.
Only recently did Hausel, Letellier and Rodriguez-Villegas prove Conjecture 2 \cite{HLRV}. 
Kac's next conjecture is significantly more far reaching.

\begin{conjecturenonum}\cite[Conjecture 3]{Ka83}
The set of isomorphism classes of indecomposable representations of $Q$ of dimension vector $\alpha$ admits a cellular decomposition by locally closed subvarieties isomorphic to affine spaces, with each $c_i$ being the number of cells of dimension $i$.
 \end{conjecturenonum}

It should be noted that the conjecture is more inspirational than literal, since the set of indecomposable isomorphism classes of a fixed dimension vector does not have a canonical structure of a variety, and generally depends on the underlying field $\kk$.
We note that for this conjecture to be true, it requires $a_\alpha(1)=\sum_i c_i$ to be the total number of cells of this cell decomposition.

We take Kac's Conjecture 3 as a major motivation for developing methods to construct cells of indecomposable representations, with aim at bringing new ideas to the classification problem of indecomposable representations.
It should be noted from the start that cell decompositions of varieties are generally far from unique and usually involve making some choices, for example of a torus action on the variety. Thus we should not expect cell decompositions of spaces of representations to canonically arise from the algebra.  Rather, we aim to develop practical methods with manageable choices that induce cell decompositions.


\subsection{Results}
In Section \ref{sec:homological} we develop recursive methods to construct cells of indecomposables in a given dimension vector from cells of indecomposables in smaller dimension vectors.  
The main idea here is to fix a representation $M$ and consider the space of self-extensions $\Ext(M,M)$ as a parameter space for deformations of $M$.  
In general, this will produce representations which are decomposable, and furthermore there will be distinct parameters which yield isomorphic representations, see \cite{Wei15}.  
We introduce the notions of \emph{strong} and \emph{separating} parameter spaces (Definition \ref{def:strongseparating}) for those which yield indecomposable and pairwise nonisomorphic representations, respectively.  
Our first main results are Theorems \ref{thm:strong} and \ref{thm:separating}, which give recursive constructions of strong and separating parameter spaces under suitable conditions.  
These can be used to produce cells of pairwise nonisomorphic indecomposables.  

In Section \ref{sec:trees} we recall the notion of \emph{tree modules}, which are quiver representations with a particularly nice basis. This can be utilized so that the application of the methods of Section \ref{sec:homological} to tree modules can be often used to derive a normal form for the deformed representations.
Tree modules are known to exist in abundance \cite{MR1090218,Ringel:1998gf,MR2578596,Kinser10,Wei12,Ringel13},
and it has been conjectured in \cite{Kinser13} that there are sufficiently many tree modules to have one in each cell in the setting of Kac's Conjecture 3. We make this more precise in Definition \ref{def:tnf} and Conjecture \ref{conj:ctnf}, supported by an example in Section \ref{sec:subspace}.
A method for recursively constructing cells of indecomposables as deformations of tree modules is given in Theorem \ref{thm:treecells}.

In Section \ref{sec:geometric} we utilize a natural torus action on moduli spaces of representations to construct cells of pairwise nonisomorphic indecomposables, when the ground field is $\C$.  We first consider a torus $(\C^*)^{|Q_1|}$ of rank equal to the number of arrows of $Q$, acting in the natural way with each copy of $\C^*$ scaling the matrices over the corresponding arrow.  
This action, also described in \cite{Wei13}, commutes with the action of the base change group on quiver representations and thus descends to the corresponding moduli space  $M^{\Theta-\mathrm{st}}_\alpha(Q)$ of stable representations (for any weight $\Theta$).  We then fix a one-dimensional subtorus $\C^* \subset (\C^*)^{|Q_1|}$ and investigate the corresponding Bia{\l}ynicki-Birula decomposition \cite{BB73} of  $M^{\Theta-\mathrm{st}}_\alpha(Q)$.  

Each fixed point of the $\C^*$-action gives a cell of pairwise nonisomorphic indecomposables (even the dimension of which depends on the choice of $\C^* \subset (\C^*)^{|Q_1|}$).  Stopping here, however, yields no concrete understanding of this cell of indecomposables, such as a normal form.  The aim of this section is to lift the cell in $M^{\Theta-\mathrm{st}}_\alpha(Q)$ to the corresponding representation variety, thus producing normal forms of the indecomposables in this cell. We point out that the points of the lifted cell can again be understood as deformations of the lifted fixed point. The major advantage of this approach is that the deformation space is automatically strong and separating.

Finally, in Section \ref{sec:applications} we demonstrate how to use and combine these methods in various applications, such as for isotropic Schur roots. As a starting point for future considerations, we also introduce certain invariants which can be attached to any root of a quiver. Actually, together with the Euler form of a root, these invariants seem to measure the complexity of the classification problem for indecomposables having this root as dimension vector.


\section{Definitions and notation}
\subsection{Quiver representations}
Here we briefly recall our definitions on quiver representations to establish notation.  More detailed background is available in many excellent textbooks, including \cite{ARS97,assemetal, schifflerbook, DWbook}.  
Let $\kk$ be a field and $Q$ be a quiver with vertices $Q_0$ and arrows $Q_1$.  Functions $s, t\colon Q_1 \to Q_0$ give source and target of an arrow $s(a) \xto{a} t(a)$.  
A \emph{representation of $Q$} over $\kk$ is denoted by $M=((M_q)_{q\in Q_0}, (M_a)_{a\in Q_1})$ where $M_q$ is a finite-dimensional $\kk$-vector space for each $q \in Q_0$, and $M_a\colon M_{s(a)} \to M_{t(a)}$ is a $\kk$-linear map for each $a\in Q_1$.
A \emph{morphism} between representations $\phi\colon M \to N$ is a collection of $\kk$-linear maps $\phi = (\phi_q\colon M_q \to N_q)_{q \in Q_0}$ satisfying $\phi_{t(a)} M_a = N_a \phi_{s(a)}$ for every $a \in Q_1$.  We write $\Hom_Q(M,N)$ or just $\Hom(M,N)$ for the $\kk$-vector space of morphisms between two representations.
We denote by $\rep_\kk(Q)$ the abelian, $\kk$-linear category of finite-dimensional $\kk$-representations of $Q$, or simply $\rep(Q)$ when $\kk$ is understood.  We write $\Ext(M,N)$ for $\Ext_{\rep(Q)}^1(M,N)$.
The \emph{dimension vector} of $M\in \rep(Q)$ is $\dimv M= (\dim_\kk M_q)_{q \in Q_0} \in \ZZ_{\geq 0}^{Q_0}$. We sometimes write $\alpha=\sum_{q\in Q_0}\alpha_qq$ for a dimension vector $\alpha\in\ZZ_{\geq 0}^{Q_0}$. On $\Z Q_0$ we have a non-symmetric bilinear form, called Euler form, defined by 
\begin{equation}\sk{\alpha}{\beta}=\sum_{q\in Q_0}\alpha_q\beta_q-\sum_{a\in Q_1}\alpha_{s(a)}\beta_{t(a)}.\end{equation}
By $\mathrm{Ind}(Q,\alpha)$ we denote the set of indecomposable representations with dimension vector $\alpha$.
 
\subsection{Parametrizing extensions}\label{sec:extensions}
For $M=(M_q)_{q \in Q_0}$ and $N=(N_q)_{q \in Q_0}$ two collections of finite-dimensional $\kk$-vector spaces,
we write
 \begin{equation}
R(N,M) := \bigoplus_{a\in Q_1}\Hom_\kk(N_{s(a)},M_{t(a)}).
 \end{equation}
Note that when $M, N \in \rep(Q)$, this does not depend on the maps $M_a$, $N_a$.

Now consider the linear map  
\begin{equation}
\begin{split}
d_{N,M}\colon\bigoplus_{q\in Q_0}\Hom_\kk(N_q,M_q)&\to R(N,M),\\
(f_q)_{q\in Q_0}&\mapsto(f_{t(a)}\,N_{a}-M_a\,f_{s(a)})_{a\in Q_1}.
\end{split}
\end{equation}
It is well known (e.g. \cite[Proposition 2.4.2]{DWbook}) that $\mathrm{ker}(d_{N,M})=\Hom(N,M)$ and $\mathrm{coker}(d_{N,M})\cong\Ext(N,M)$. Thus we have
\begin{equation}\label{eq:Euler}
\sk{N}{M}:=\sk{\udim N}{\udim M}=\dim\Hom(N,M)-\dim\Ext(N,M).
\end{equation}
Let 
\begin{equation}\label{eq:piNM}
\pi_{N,M} \colon R(N,M) \to \Ext(N,M)
\end{equation}
be the natural projection.  Concretely, an element $f =(f_a)_{a \in Q_1} \in R(N,M)$ determines the short exact sequence $\pi_{N,M}(f)\colon\sesB{M}{B(f)}{N}{\iota}{\pi}$, where the middle term is defined by the vector spaces $B(f)_q=M_q\oplus N_q$ for all $q\in Q_0$ and linear maps $B(f)_a=\begin{pmatrix}M_a &f_a\\0&N_a\end{pmatrix}$ for all $a\in Q_1$.  

\begin{definition}\label{def:representssubset}
We say a subset $U \subset R(N,M)$ \emph{represents a subset $E\subseteq \Ext(N,M)$} if the restriction of $\pi_{N,M}$ to $U$ gives a bijection of $U$ with $E$. If $E$ is a basis of $\Ext(N,M)$, we say that $U$ \emph{represents a basis} of $\Ext(N,M)$.
\end{definition}

Given a dimension vector $\alpha$ for $Q$, the associated \emph{representation variety} is 
\begin{equation}
R_\alpha(Q):= \bigoplus_{a\in Q_1}\Hom_\kk(\kk^{\alpha_{s(a)}},\kk^{\alpha_{t(a)}}).
\end{equation}
If $M\in\rep(Q)$ with $M_q = \kk^{\alpha_q}$ is of dimension vector $\alpha$, the spaces $R(M,M)$ and $R_\alpha(Q)$ are by definition the same. We regard $R_\alpha(Q)$ both as vector space and as an affine variety, depending on the context.  We use the former to emphasize that its points represent self-extensions or deformations (as defined in Section \ref{sec:paramreps}) of $M$, and the latter to emphasize its points correspond to all representations of dimension vector $\alpha$.

\subsection{Deformations of representations}\label{sec:paramreps}
Fix a point $M \in R_\alpha(Q)$. Each $\lambda \in R(M,M)$ defines a representation
\begin{equation}
M(\lambda) := M+ \lambda
\end{equation}
of the same dimension vector, with the sum taken in the vector space $R(M,M)$. This notation emphasizes that we think of $M(\lambda)$ as a deformation of $M$ by the parameter $\lambda$. 

\begin{definition}\label{def:strongseparating}
Let $M \in \rep(Q)$ and $U \subset R(M,M)$ a subset.
\begin{enumerate}[(1)]
\item We call $U$ \emph{strong} if $M(\lambda)$ is indecomposable for every $\lambda\in U$.
\item We call $U$ \emph{separating} if $M(\lambda)\ncong M(\mu)$ for every $\lambda,\mu\in U$ with $\lambda\neq\mu$.
\end{enumerate}
\end{definition}

\begin{example}\label{ex:kronecker}
Consider the modules of dimension vector $(1,1)$ for the (generalized) Kronecker quiver $K(n)=(\{0,1\},\{a_1,\ldots,a_n\})$ with $s(a_i)=0$ and $t(a_i)=1$ for $i=1,\ldots,n$.
This very simple example serves to illustrate the notation and terminology.

Starting with the two simple modules $S_0$ and $S_1$ of dimension $(1,0)$ and $(0,1)$ respectively, we have 
\begin{equation}
R(S_0,S_1) \cong \kk^{Q_1} \cong \Ext(S_0,S_1).
\end{equation}
We denote by $0\xto{a_i}1$ the vector of $R(S_0,S_1)$ which is $1$ in coordinate $a_i$ and 0 elsewhere. The set
\begin{equation}\cR_{S_0,S_1}=\{0\xto{a_i}1\mid i=1,\ldots,n\}\end{equation}
 represents a basis of $\Ext(S_0, S_1)$. The corresponding exact sequence has middle term $T_i$ of dimension $(1,1)$ which is the representation visualized in the same way: $T_{i}:=0\xto{a_i}1$.
For each $i$, the subset
\begin{equation}\cR_{T_i}=\{0\xto{a_j} 1\mid j\neq i\}\end{equation}
represents basis of $\Ext(T_i,T_i)$ such that $\langle \cR_{T_i}\rangle\subset R(T_i,T_i)$ is strong and separating.

Each $\lambda\in\langle \cR_{T_i}\rangle$ defines a deformation $T_i(\lambda)$ of $T_i=T_i(0)$ by
\begin{equation}T_i(\lambda)=((\kk,\kk),([\lambda_1],\ldots,[\lambda_{i-1}],[1],[\lambda_{i+1}],\ldots,[\lambda_n]),\end{equation}
thus $\langle \cR_{T_i}\rangle\subset R_{(1,1)}(K(n))$ is an $(n-1)$-dimensional affine cell of nonisomorphic indecomposables of dimension vector $(1,1)$. 

Note that by elementary considerations one can see that the isomorphism classes of $K(n)$ of dimension vector $(1,1)$ are parametrized by $\Pn^{n-1}$. Each of them can be constructed as a deformation in our language above, and we can make it so that each isomorphism class appears in exactly one of our cells by shrinking $T_i(\lambda)$ to the $(n-i)$-dimensional cell with $\lambda_1=\ldots=\lambda_{i-1}=0$.
\end{example}

Let $M \in R_\alpha(Q)$ and $N \in R_\beta(Q)$ be representations of $Q$.
As $\lambda \in R(M,M)$ and $\mu \in R(N,N)$ vary, we cannot naturally identify the spaces $\Ext(N(\mu), M(\lambda))$ (their dimension may vary even).  
However, we have for each $(\mu, \lambda)$ the surjective map
\begin{equation}\label{eq:pilm}
\pi_{\mu,\lambda}:=\pi_{N(\mu), M(\lambda)} \colon R(N,M) \onto \Ext(N(\mu), M(\lambda)), 
\end{equation}
defined in \eqref{eq:piNM} which we can use to compare these spaces by working with representatives in $R(N,M)$. 

Each triple $(\tau,\lambda,\mu)\in R(N,M) \times R(M,M)\times R(N,N)$  defines a short exact sequence 
$\pi_{\mu,\lambda}(\tau)\in\Ext(N(\mu),M(\lambda))$ with middle term $B(\tau,\lambda,\mu)$ given by
\begin{equation}B(\tau,\lambda,\mu)_q:=M_q\oplus N_q,\quad B(\tau,\lambda,\mu)_a:=
\begin{pmatrix}M(\lambda)_a&\tau_a\\0&N(\mu)_a\end{pmatrix}\end{equation}
for every $q\in Q_0$ and $a\in Q_1$ respectively.

\begin{definition}\label{def:universal}
Let $U\subseteq R(M,M)$ and $V\subseteq R(N,N)$ be subsets.
\begin{enumerate}[(1)]
\item A subset $W \subseteq R(N,M)$ is called \emph{universal for the pair} $(U,V)$ if $W$ represents a subset of $\Ext(N(\mu),M(\lambda))$ for all $(\lambda, \mu)\in U \times V$.\item If $\cR\subset R(N,M)$ represents a basis of $\Ext(N(\mu),M(\lambda))$ for all $(\lambda,\mu)\in U\times V$, we say that $\cR$ \emph{is a universal basis for} $(U,V)$.
\end{enumerate}
\end{definition}

For instance, if $\im d_{N,M}=\im d_{N(\mu),M(\lambda)}$ for all $(\lambda,\mu)\in U\times V$, every $\cR \subset R(N,M)$ representing a basis of $\Ext(N,M)$, is already a universal basis for $(U,V)$. A special case where this occurs is $\mathrm{supp}(M)\cap\mathrm{supp}(N)=\emptyset$, since both images are 0.

\section{Homological construction of families of indecomposables}\label{sec:homological}
\noindent Throughout this section we fix a quiver $Q$ and a field $\kk$.

\subsection{Extensions of indecomposables}
We start with a criterion for an extension of indecomposables to be indecomposable. 

\begin{lemma}\label{lem:sesmorphism}Let $M, N, M',N'\in\rep(Q)$. 
Consider a pair of short exact sequences 
\begin{equation*}
0 \to M \xto{\iota} B \xto{\pi} N \to 0
\quad \text{and} \quad
0 \to M' \xto{\iota'} B' \xto{\pi'} N' \to 0.
\end{equation*}
If the induced map $\Hom(B,B')\to \Hom(M,N')$ given by $\phi \mapsto \pi'\circ \phi\circ\iota$ is the zero map, then each $\phi \in \Hom(B,B')$ induces a morphism of short exact sequences
\begin{center}
  \begin{tikzcd}%
        0 \arrow[r] & M \arrow[r, "\iota"] \arrow[d, "\phi|_M"] & B \arrow[r, "\pi"]  \arrow[d, "\phi"] & N \arrow[r] \arrow[d, "\bar{\phi}"]  & 0 \\  
        0 \arrow[r] & M' \arrow[r, "\iota' "]  & B' \arrow[r, "\pi' "]  & N' \arrow[r]  & 0 
    \end{tikzcd}%
\end{center}
\end{lemma}
\begin{proof}
Since $\pi'\circ\phi\circ\iota=0$ by assumption, the universal property of $\ker \pi'$ gives a factorization of $\phi\circ \iota$ through $M'$, inducing the commutative square at left. Furthermore, $\pi'\circ \phi$ vanishes on the image of $\iota$, so it factors through top $\pi$ by the universal property of the cokernel of $\iota$, inducing the commutative square at right.
\end{proof}

Recall that a finite-dimensional quiver representation $M$ is indecomposable if and only if its endomorphism ring is \emph{local}; two equivalent characterizations of this property in our setting are that every element of $\End(M)$ is either an isomorphism or nilpotent, and that the only idempotents of $\End(M)$ are 0 and 1.
See \cite[Ch.1\S1]{LWbook} for an exposition in sufficient generality (in particular, no hypotheses on $\kk$ or $Q$ are necessary).

\begin{lemma}\label{lem:indecomposable1} Let $M,N\in\rep(Q)$ be indecomposable and let
\begin{equation*}\sesB{M}{B}{N}{\iota}{\pi}\end{equation*} be a nonsplit short exact sequence.  If the induced map $\End(B)\to \Hom(M,N)$ sending $\phi \mapsto \pi\circ \phi\circ\iota$ is the zero map, then $B$ is indecomposable.
\end{lemma}
\begin{proof}
We will show that $\End(B)$ is local.  
By Lemma \ref{lem:sesmorphism}, every $\phi \in \End(B)$ induces a morphism of short exact sequences; in particular, we have a $\kk$-algebra homomorphism $\Psi:\End(B)\to\End(M)$ where $\Psi(\phi) = \phi|_M$.  
We show that every element of $\ker\Psi$ is nilpotent.  If  $\phi\in\ker\Psi$, we obtain a commutative diagram
\begin{equation}\xymatrix{0\ar[r]&M\ar[d]^0\ar^{\iota}[r]&B\ar^{\pi}[r]\ar^{\phi}[d]&N\ar_{\tau}[ld]\ar^{\bar{\phi}}[d]\ar[r]&0\\0\ar[r]&M\ar^{\iota}[r]&B\ar^{\pi}[r]&N\ar[r]&0}\end{equation}
where the factorization $\phi=\tau\circ\pi$ arises from $\phi\circ \iota = 0$ by the universal property of $\coker \iota$. This universal property also gives that $\bar\phi$ is unique with $\bar\phi\circ\pi=\pi\circ\phi$. It follows that $\phi=\tau\circ\pi$ already gives $\bar{\phi}=\pi\circ \tau$ in the ring $\End(N)$, which is local since $N$ is indecomposable.  If $\bar{\phi}$ were a unit, the sequence would split as $\pi\circ(\tau\circ\bar{\phi}^{-1})=\mathrm{id}_N$ in this case, contradicting our assumption.  Therefore, $\bar{\phi}$ is nilpotent and there exists a positive integer $n$ such that $\bar\phi^n=0$. Then 
\begin{equation}\phi^{n+1}=(\tau\circ \pi)^{n+1}=\tau\circ\bar{\phi}^n\circ\pi=0,\end{equation}
showing that $\phi$ is nilpotent as well. Thus every element in $\ker\Psi$ is nilpotent.

Now take an arbitrary idempotent $e \in \End(B)$ and the associated morphism of short exact sequences.
\begin{equation}
  \begin{tikzcd}%
        0 \arrow[r] & M \arrow[r, "\iota"] \arrow[d, "e|_M"] & B \arrow[r, "\pi"]  \arrow[d, "e"] & N \arrow[r] \arrow[d, "\bar{e}"]  & 0 \\  
        0 \arrow[r] & M \arrow[r, "\iota "]  & B \arrow[r, "\pi "]  & N \arrow[r]  & 0 
    \end{tikzcd}%
\end{equation}
Since $\Psi$ is an algebra homomorphism, it preserves idempotents. Thus since $\End(M)$ is local, we have $e|_M=0$ or $e|_M=1$.  If $e|_M=\Psi(e)=0$, then $e$ is idempotent and nilpotent by the previous paragraph, so $e=0$.  So we can assume $e|_M=1$.  Similarly, either $\bar{e}=0$ or $\bar{e}=1$. If $\bar{e}=0$, then $e$ would be an idempotent of $\End(B)$ with image $M$, thus $M$ would be a direct summand of $B$, contradicting our assumption that the sequence does not split.  So we can assume $\bar{e}=1$.  But then $e$ is an invertible idempotent, so $e=1$.
We have shown that the only idempotents of $\End(B)$ are 0 and 1, so $\End(B)$ is local.
\end{proof}

\subsection{Families of indecomposable extensions}
Below, we will frequently use that for an extension $B$ of a representation $N$ by $M$, a vector space decomposition $B_q = M_q \oplus N_q$ at each $q \in Q_0$ induces a decomposition
\begin{equation}\label{eq:RBB}
R(B,B) \cong R(N,N) \oplus R(N,M) \oplus R(M,N) \oplus R(M,M).
\end{equation}
With this, we can naturally associate to any  $f\in R(X,Y)$ and $X,Y\in\{M,N\}$ a self-extension $\pi_{B,B}(\iota_{X,Y}(f))$ of $B$, where $\iota_{X,Y}:R(X,Y)\to R(B,B)$ is the natural embedding.

\begin{notation} For the remainder of the section, including the statements of the theorems, we fix the following notation associated to fixed $M, N \in \rep(Q)$:
\begin{itemize}
\item a nonzero $e \in R(N,M)$ determining an extension $\ses{M}{B}{N}$;
\item for each pair $X,Y\in\{M,N\}$,  a subset $U_{X,Y}\subset R(X,Y)$, writing $U_X:=U_{X,X}$ for short.
\end{itemize}
\end{notation}

Recall from Section \ref{sec:paramreps} that every triple $(\tau,\lambda,\mu)\in U_{N,M}\times U_{M}\times U_{N}$ gives rise to a short exact sequence
\begin{equation}\label{eq:taulambdamuseq}\
\quad \ses{M(\lambda)}{B(e+\tau,\lambda,\mu)}{N(\mu)}.
\end{equation}
The sequence is nonsplit if and only if $\pi_{\mu,\lambda}(e+\tau) \neq 0$.

As in Lemma \ref{lem:sesmorphism}, for all pairs $(\tau,\lambda,\mu),(\tau',\lambda',\mu')\in U_{N,M}\times U_{M}\times U_{N}$, we have a linear map
\begin{equation}\Theta_{e+\tau,\lambda,\mu}^{e+\tau',\lambda',\mu'}\colon\Hom(B(e+\tau,\lambda,\mu),B(e+\tau',\lambda',\mu'))\to\Hom(M(\lambda),N(\mu'))\end{equation}
given by precomposition with the inclusion $M(\lambda) \into B(e+\tau,\lambda,\mu)$ followed by postcomposition with the surjection $B(e+\tau',\lambda',\mu')\onto N(\mu')$.  

\begin{remark}
While the condition that $\Theta_{e+\tau,\lambda,\mu}^{e+\tau',\lambda',\mu'}=0$ below looks somewhat technical, it is often easy to verify in practice.  For example, it obviously holds if $\Hom(M(\lambda), N(\mu'))=0$, which happens for all $(\lambda, \mu')\in U_M\times U_N$ when $M$ and $N$ have disjoint support.
\end{remark}

In the theorems below, we use the following notation.  If $V$ is a vector space, $v\in V$, and $V'\subset V$ a subset, we write $v+V'=\setst{v+v'}{v' \in V'}$. 
We now apply Lemmas \ref{lem:sesmorphism} and \ref{lem:indecomposable1} to build families of indecomposables with nice properties.

\begin{theorem}\label{thm:strong}
Assume that both  $U_M$ and $U_N$ are both strong, and that $W \subseteq U_{N,M} \times U_M \times U_N$ is a subset such that:
\begin{itemize}
\item $\pi_{\mu,\lambda}(e+\tau) \neq 0$ for all $(\tau, \lambda, \mu) \in W$;
\item $\Theta_{e+\tau,\lambda,\mu}^{e+\tau,\lambda,\mu}=0$ for all $(\tau,\lambda,\mu) \in W$.
\end{itemize}
Then under the identification \eqref{eq:RBB}, the subset $(e,0,0)+W \subset R(B,B)$ is strong.
\end{theorem}
\begin{proof}
Each element $(\tau,\lambda,\mu)\in U_{N,M}\times U_M\times U_N$ determines a short exact sequence
\begin{equation}\label{eq:strongseq}
\pi_{\mu,\lambda}(e+\tau):\ses{M(\lambda)}{B(e+\tau,\lambda,\mu)}{N(\mu)}.
\end{equation}
Since $U_M$ and $U_N$ are each assumed to be strong, both $M(\lambda)$ and $N(\mu)$ are indecomposable.
As each $\pi_{\mu,\lambda}(e+\tau) \neq 0$, each sequence \eqref{eq:strongseq} does not split. 
From the assumption that  $\Theta_{e+\tau,\lambda,\mu}^{e+\tau,\lambda,\mu}=0$ for all $(\tau,\lambda,\mu) \in W$, Lemma \ref{lem:indecomposable1} implies that each $B(e+\tau,\lambda,\mu)$ is indecomposable and thus $(e,0,0)+W$ is strong.
\end{proof}


\begin{theorem}\label{thm:separating}
Assume the following:
\begin{itemize}
\item $U_M$ and $U_N$ are both separating;
\item $\Theta_{e+\tau,\lambda,\mu}^{e+\tau',\lambda',\mu'}=0$ for all $(\tau,\lambda,\mu), (\tau',\lambda',\mu') \in U_{N,M} \times U_M \times U_N$.
\end{itemize}  
Then we have:

\begin{enumerate}[(a)]
\item Under the identification \eqref{eq:RBB}, the set $\{e\}\times U_M \times U_N \subset R(B,B)$ is always separating.
\item If furthermore $\End(M(\lambda)) = \End(N(\mu))=\kk$, the element $e$ is not in the subspace of $R(N,M)$ generated by $U_{N,M}$, and the set $e+U_{N,M}$ is universal for $(U_M, U_N)$, then $(e+U_{N,M}) \times U_M \times U_N \subset R(B,B)$ is separating.
\end{enumerate}
\end{theorem}
\begin{proof}
Given $(\tau, \lambda, \mu), (\tau', \lambda', \mu') \in U_{N,M} \times U_M \times U_N$, Lemma \ref{lem:sesmorphism} shows that any isomorphism $\phi\colon B(e+\tau,\lambda,\mu)\xto{\sim} B(e+\tau',\lambda',\mu')$ induces the morphism of exact sequences below.

\begin{equation}\xymatrix@R24pt@C30pt{0\ar[r]&M(\lambda)\ar^{\phi|_{M(\lambda)}}[d]\ar^{\iota\quad}[r]&B(e+\tau,\lambda,\mu)\ar^{\quad\pi}[r]\ar^{\phi}[d]&N(\mu)\ar^{\bar{\phi}}[d]\ar[r]&0\\
0\ar[r]&M(\lambda')\ar^{\iota'\quad}[r]&B(e+\tau',\lambda',\mu')\ar^{\quad\pi'}[r]&N(\mu')\ar[r]&0}\end{equation}
By dimension count we have that both $\phi|_{M(\lambda)}$ and $\bar{\phi}$ are isomorphisms, which yields $\lambda=\lambda'$ and $\mu=\mu'$ because $U_M$ and $U_N$ are assumed to be separating.  Taking $\tau=0$, this proves (a).  

For part (b), it remains to show that the additional assumptions imply $\tau=\tau'$. Consider the structure of the morphism $\phi$ with respect to the vector space decomposition $B= M\oplus N$, i.e.
\begin{equation}\phi=\begin{pmatrix}\phi_{1,1}&\phi_{1,2}\\\phi_{2,1}&\phi_{2,2}\end{pmatrix}.\end{equation}
The commutativity of the diagram yields $\phi_{2,1}=0$. Together with our assumption on triviality of endomorphism rings this gives $\phi_{1,1}=\phi|_{M(\lambda)}=c\cdot \mathrm{id}_{M(\lambda)}$ and $\phi_{2,2}=\bar{\phi}=d\cdot \mathrm{id}_{N(\mu)}$ for certain $c,d\in\kk^\ast$. 

In turn, $\phi_{t(a)}\circ B(e+\tau,\lambda,\mu)_a=B(e+\tau',\lambda',\mu')_a\circ \phi_{s(a)}$ for 
\begin{equation}
B(e+\tau,\lambda,\mu)_a=\begin{pmatrix}M(\lambda)_a&e_a+\tau_a\\0&N(\mu)_a\end{pmatrix}
 \quad \text{and} \quad
B(e+\tau',\lambda',\mu')_a=\begin{pmatrix}M(\lambda')_a&e_a+\tau'_a\\0&N(\mu')_a\end{pmatrix},
\end{equation}
gives \begin{equation}(e_a+\tau'_a)\cdot d+M(\lambda')_a\circ (\phi_{1,2})_{s(a)} =c\cdot \left(e_a+\tau_a \right)+(\phi_{1,2})_{t(a)}\circ N(\mu)_a\end{equation} 
for all $a\in Q_1$ and thus
\begin{equation}d_{N(\mu),M(\lambda')}\left((\phi_{1,2})_a\right)=\left(e_a(d-c)+\tau'_a\cdot d-\tau_ac\right)_a.\end{equation}
Now recall that the image of $d_{N(\mu),M(\lambda')}$ is exactly the kernel of $\pi_{\mu, \lambda'}$ defined in \eqref{eq:pilm}, so we have $\pi_{\mu,\lambda'}(\left(e_a(d-c)+\tau'_a\cdot d-\tau_ac\right)_a) = 0$.  On the other hand, we assumed that $e+U_{N,M}$ represents a subset of $\Ext(N(\mu),M(\lambda'))$, meaning that $\pi_{\mu,\lambda'}$ is injective on $e+U_{N,M}$, so we find that $\left(e_a(d-c)+\tau'_a\cdot d-\tau_ac\right)_a = 0$ in $R(N,M)$.  But we also assumed that $e$ is not in the subspace of $R(N,M)$ generated by $U_{N,M}$, so we must have $d=c$ and $\tau=\tau'$.
\end{proof}

\begin{remark} 
The assumption on $M(\lambda)$ and $N(\mu)$ to have trivial endomorphism ring is necessary in (b) because otherwise there are cases where two middle terms $B(\tau,\lambda,\mu)$ and $B(\tau',\lambda,\mu)$ for $\tau\neq\tau'$ are isomorphic. 
\end{remark}

Relevant to the conjecture of Kac discussed in the introduction is the following definition, where we have the deformation parameter varying over a linear subspace of $R(M,M)$. 

\subsection{Affine cells of indecomposable representations}
Recall that a representation $M$ is called \emph{Schurian} if $\End(M)=\kk$.

\begin{definition} We call a pair $(M,U)$ consisting of a representation $M$ of $Q$ and a subspace $U\subseteq R(M,M)$, which represents a strong and separating subset of $\Ext(M,M)$, a \emph{cell of indecomposable representations}. If all representations $M(\lambda)$ for $\lambda\in U$ are Schurian, we call $(M,U)$ a \emph{cell of Schurian representations}.

Let $M_1,\ldots, M_n\in\rep(Q)$ with the same dimension vector and fix $U_i\subset R(M_i,M_i)$ for $i=1,\ldots,n$. If each $(M_i,U_i)$ is a cell of indecomposable representations such that $M_i(\lambda_i)\cong M_j(\lambda_j)$ implies $i=j$, we call the collection $\{(M_i,U_i)\mid i=1,\ldots,n\}$ a \emph{mosaic of indecomposable representations}.
\end{definition}
Note that in a mosaic of indecomposable representations, $M_i(\lambda_i)\cong M_i(\lambda_j)$ already implies $\lambda_i=\lambda_j$ as each $U_i$ is separating.

\begin{example}\label{ex:grassmannian}
Choosing a basis, a representation of dimension $(1,d)$ of $K(n)$ is given by a matrix $A\in\kk^{d\times n}$. Moreover, $A$ is indecomposable if and only if $\rank(A)=d$, and $A\cong A'$ if and only if there exists a $g\in\Gl_d$ such that $gA=A'$. This shows that the isomorphism classes of indecomposable representations are in one-to-one correspondence with points of the usual Grassmannian $\Gr_d(\kk^n)$. So the Schubert decomposition induces a mosaic of indecomposable representations for $(1,d)$.

More precisely, recall that for every sequence $I=(i_1,\ldots,i_d)$ with $1\leq i_1<\ldots<i_d\leq n$ there exists a Schubert cell
\begin{equation}A_I:=\begin{pmatrix}
    \ast&\cdots &\ast &1&0 &\cdots& 0 & 0 & 0 &\cdots&0&0&0&\cdots&0\\
    \ast&\cdots &\ast&0&\ast&\cdots &\ast&1&0&\cdots&0&0&0&\cdots&0\\
    \ast &\cdots&\ast&0&\ast &\cdots&\ast&0&\ast&\cdots&0&0&0&\cdots&0\\[-0.4em]
    \vdots &\ddots&\vdots&\vdots&\vdots &\ddots&\vdots&\vdots&\vdots&\ddots&\vdots&\vdots&\vdots&\ddots&\vdots\\
    \ast&\cdots &\ast&0&\ast&\cdots &\ast&0&\ast&\cdots&0&0&0&\cdots&0\\
    \ast&\cdots &\ast&0&\ast&\cdots &\ast&0&\ast&\cdots&\ast&1&0&\cdots&0
  \end{pmatrix}\subset \kk^{d\times n}\end{equation}
where the unit vectors are in the columns $i_1,\ldots,i_d$. Then we have \begin{equation}\Gr_d(\kk^n)\cong\coprod_{\substack{I\subset\{1,\ldots,n\},\\|I|=d}}A_I.\end{equation}

Then $A_I$ defines an affine space of Schurian representations of dimension $(1,d)$. Furthermore, taking $B_I$ as the representation with $(B_I)_{a_{i_l}}=e_{i_l}$ and $(B_I)_{a_i}=0$ if $i\notin I$, every point of $A_I$ can be understood as a deformation of $B_I$, giving a cell of Schurian representations.

Note that the Schubert decomposition can also be obtained with the geometric methods presented in Section \ref{sec:geometric}.
\end{example}

The following proposition uses the notation for generalized Kronecker quivers of Example \ref{ex:kronecker}.

\begin{proposition}\label{pro:KnPn}
Let $M,N$ be Schurian representations with $\Hom(M,N)=0$ and assume that $\{e_1,\ldots,e_n\}\subset R(N,M)$ represents a basis of $\Ext(N,M)$. Let $U_{N,M}=\langle e_1,\ldots,e_n\rangle$. For any $1\leq d\leq n$ we have a linear map 
\begin{equation*}
\Gamma\colon R_{(1,d)}(K(n))\to U_{N,M}\otimes \kk^d,\qquad  (v_{a_i})_{i=1}^n \mapsto \sum_{i=1}^n e_i\otimes v_{a_i}
\end{equation*}
which determines for each $A=(v_{a_i})_{i=1}^n \in R_{(1,d)}(K(n))$ an extension of $N$ by $M\otimes \kk^d \cong M^{d}$ with middle term denoted by $F(A)$.

Then $\Gamma(A),\Gamma(A')\in U_{N,M}\otimes \kk^d$ give isomorphic middle terms if and only if there exists a $g\in\Gl_d$ such that $A= gA'$, i.e. if $A\cong A'$ as representations of $K(n)$. Furthermore, $F(A)$ is indecomposable if and only if $\rank(A)=d$, i.e. if $A\in R_{(1,d)}(K(n))$ is indecomposable. 
The analogous statement also holds for the map $\Gamma':R_{(d,1)}(K(n))\to U_{N,M}\otimes \kk^d$ defined in the natural way, i.e. when considering extensions of $N^d$ by $M$.
\end{proposition}
\begin{proof}
Fix $A,A'\in\kk^{d\times n}$ and let $\phi \in\Hom_Q(B,B')$ with $B=F(A)$ and $B'=F(A')$. We write 
$\phi=\begin{pmatrix}\phi_{1,1}&\phi_{1,2}\\ \phi_{2,1}&\phi_{2,2}\end{pmatrix}$ with linear maps $\phi_{i,j}$ determined by the vector space decomposition $B=N \oplus M\otimes \kk^d$ and similarly for $B'$. 
Making $\phi_{t(a)}\circ B_a=B'_a\circ \phi_{s(a)}$ for every $a\in Q_1$ explicit, we obtain $\phi_{2,1}=0$ using $\Hom(M\otimes\kk^d,N)=0$, then $\phi_{2,2}=\mu\id_N$ for some $\mu\in\kk$ using $\End(N)=\kk$, and then $\phi_{1,1}=\id_M \otimes g \in\End(M\otimes\kk^d)\cong \End(M) \otimes \Hom_{\kk}(\kk^d, \kk^d)$ using $\End(M)=\kk$, where $g \in \Hom_{\kk}(\kk^d, \kk^d)$.

Using the natural isomorphism $U_{N,M\otimes\kk^d}\cong U_{N,M}\otimes\kk^d$, we furthermore obtain for every $a\in Q_1$ that
\begin{equation}(\phi_{1,1})_{t(a)}\circ \Gamma(A)_a +(\phi_{1,2})_{t(a)}\circ N_a=(M_a \otimes \id_{\kk^d})\circ (\phi_{1,2})_{s(a)}+\Gamma(A')_a \circ (\phi_{2,2})_{s(a)}\end{equation}

which yields the following after substitution, rearranging, and collecting terms over all $a \in Q_1$:
\begin{equation}d_{N,M\otimes\kk^d}(\phi_{1,2})=\mu \Gamma(A')- \Gamma(g\cdot A) \in U_{N,M\otimes\kk^d}.\end{equation}

 As $U_{N,M}$ represents a subspace of $\Ext(N,M)$, it follows that $U_{N,M}\otimes \kk^d$ represents a subspace of $\Ext(N,M\otimes\kk^d)$. Thus $\im d_{N,M\otimes\kk^d}=0$ in $\Ext(N,M\otimes \kk^d)\cong \Ext(N,M)\otimes \kk^d$ gives $\mu\Gamma(A')= \Gamma(g\cdot A) $.  

Since $\{e_1, \dotsc, e_n\}$ is linearly independent, writing $A=(v_{a_i})_{i=1}^n$ and $A'=(v'_{a_i})_{i=1}^n$, we get
\begin{equation}
\sum_{i=1}^n e_i\otimes \mu v'_{a_i} = \sum_{i=1}^n e_i\otimes g\, v_{a_i} \Rightarrow \mu v'_{a_i} = g\, v_{a_i} \qquad 1 \leq i \leq n,
\end{equation}
which is exactly the condition that $(\mu,g)\in\Hom_{K(n)}(A,A')$.  It is immediate from the triangular form of $\phi$ and descriptions of $\phi_{1,1}, \phi_{2,2}$ above that $\phi$ is an isomorphism if and only if $(\mu, g)$ is an isomorphism.
Since we started with an arbitrary $\phi\in\Hom_Q(B,B')$, we see that $B \cong B'$ implies $A \cong A'$.
On the other hand, every $(\mu,g)\in\Hom_{K(n)}(A,A')$ gives a morphism $\phi$ with $\phi_{1,2}=\phi_{2,1}=0$ in the above notation.  So $A \cong A'$ implies $B \cong B'$ as well.

Now we see that $A$ is indecomposable if and only if $B$ is indecomposable.
If $A$ is decomposable and $(g,\mu)$ a nontrivial idempotent, then $\phi$ as above with $\phi_{1,2}=0$ defines a nontrivial idempotent of $B$.
Conversely, let $\phi\in\End(B)$ as above an idempotent and $A$ be indecomposable. The triangular form of $\phi$ gives 
\begin{equation}g^2=g,\qquad \mu^2=\mu,\qquad (\id_M\otimes g)_q \circ (\phi_{1,2})_q+\mu(\phi_{1,2})_q=(\phi_{1,2})_q\end{equation}
for every $q\in Q_0$. Thus $(g,\mu)$ is an idempotent of $\End_{K(n)}(A)$ and it follows that $(g,\mu)=(0,0)$ or $(g,\mu)=(\id_{\kk^d},1)$. In the first case it follows that $\phi_{1,2}=0$ and thus $\phi=0$. In the second case, we get $\phi_{1,2}+\phi_{1,2}=\phi_{1,2}$ and thus again $\phi_{1,2}=0$ which gives $\phi=\id_B$. This shows that $\phi=0$ or $\phi=\id_{B}$.
\end{proof}

\begin{remark}Note that if $\Hom(N,M)\neq 0$, the representations $F(A)$ are not Schurian. This can be checked straightforwardly, but it is also revealed by the proof as it shows that $\phi_{1,2}\in\Hom_Q(N,M\otimes \kk^d)$ for $\phi\in\End(B)$. 

Actually, $F$ can be extended to a faithful functor $F:\rep(K(n))\to \rep(Q)$ which is full if $\Hom(N,M)=0$. In this case, it automatically reflects indecomposables and isomorphisms. For exceptional representations $M$ and $N$, this follows directly from Schofield induction \cite{Sch91}. Moreover, the results of \cite[Section 3.2]{Wei15} apply in the situation of Proposition \ref{pro:KnPn}. Since the result as stated has an easier proof, we included it -  also to make the paper self-contained.

\end{remark}

\begin{remark}\label{rmk:grasscells}
 Now Proposition \ref{pro:KnPn} gives a parametrization of isomorphism classes of extensions of certain Schurian representations by the variety $\Gr_d(\kk^n)$. In particular, the Schubert decomposition into affine spaces  can be carried over to construct a mosaic of indecomposables of $Q$ of dimension $\udim N +d\cdot\udim M.$ 

More explicitly, let $\{e_1,\ldots,e_{n}\}\subset U_{N,M}$ represent a basis of $\Ext(N,M)$. Each of the ${n\choose d}$ strictly increasing sequences $I=(i_1,\ldots,i_d)$ gives an indecomposable representation $B_I$ as the middle term of the exact sequence $\pi_{N,M\otimes \kk^d}(\sum_{j=1}^d e_{i_j}\otimes f_j)$, where $(f_1, \dotsc, f_d)$ is the standard basis of $\kk^d$. The corresponding cell of indecomposables is then given by
\begin{equation}C_I=\left(B_I, \left\{\sum_{j=1}^d \sum_{k=1}^{i_j-1}a_{j,k}e_k \otimes f_j \mid a_{j,k}\in\kk,\,a_{l,i_k}=0\text{ for }k,l=1,\ldots,d\right\}\right).\end{equation}
Thus we obtain a natural embedding $\iota_I\colon A_I\hookrightarrow U_{N,M\otimes \kk^d}\hookrightarrow R(B_I,B_I)$ in such a way that $A_I$ represents a strong and separating subspace of $\Ext(B_I,B_I)$.

Furthermore, the collection of cells $\{C_I \mid I=(i_1,\ldots,i_d),\,1\leq i_1<\ldots<i_d\leq n\}$ is a mosaic of indecomposables.

\end{remark}


The preceding remark can be generalized when allowing the outer terms of the extension to vary within cells of indecomposables. If $M\in\rep(Q)$, we can identify $R(M\otimes \kk^d,M\otimes \kk^d) \cong R(M,M) \otimes \Hom(\kk^d, \kk^d)$,
and if $U_M\subset R(M,M)$ represents a subspace of $\Ext(M,M)$, then for $\lambda\in U_M$ we have $(M\otimes \kk^d)(\lambda \otimes \id_{\kk^d})=M(\lambda)\otimes \kk^d.$

\begin{theorem}\label{thm:cells}
Let $(M,U_M)$ and $(N,U_N)$ be cells of Schurian representations such that we have $\Hom(M(\lambda),N(\mu))=0$ for all $(\lambda,\mu)\in U_M\times U_N$ and let $U_{N,M}\subset R(N,M)$ be a subspace which is universal for $(U_N,U_M)$. Let $n=\dim U_{N,M}$ and assume that $\cR_{N,M}=\{e_1,\ldots,e_n\}$ is a basis of $U_{N,M}$. Then there exists a $(\Gr_{d}(\kk^n)\times U_M\times U_N)$-family of nonisomorphic indecomposable representations of dimension $d\cdot \udim M+\udim N$ such that, under the identification \eqref{eq:RBB}, 
\begin{equation*}C_I=\left(B_I,\left\{\left(\tau,\lambda\otimes\id_{\kk^d},\mu\right)\mid \tau\in \iota_I(A_I),\,\lambda\in U_M,\mu\in U_N\right\}\right)\end{equation*}
is a cell of indecomposables with  $\dim C_I=\dim U_M+\dim U_N+\dim A_I$ for every strictly increasing $I=(i_1,\ldots,i_d)$. Moreover, $\{C_I\}_I$ is a mosaic of indecomposables.
\end{theorem}
\begin{proof} As $(M,U_M)$ and $(N,U_N)$ are cells of Schurian representations with $\Hom(M(\lambda),N(\mu))=0$ for $(\lambda,\mu)\in U_M\times U_N$ and as $U_{N,M}$ is universal for $(U_N,U_M)$, we can apply Proposition \ref{pro:KnPn} to every pair $M(\lambda),\,N(\mu)$ which shows that every representation $B_I(\tau,\lambda\otimes\id_{\kk^d},\mu)$ with $\tau\in \iota_I(A_I)$ is indecomposable.

Furthermore, if $B_I(\tau,\lambda\otimes\id_{\kk^d},\mu)\cong B_{I'}(\tau',\lambda'\otimes\id_{\kk^d},\mu')$, analogously to the proof of Theorem \ref{thm:separating}, we obtain $\lambda=\lambda'$ and $\mu=\mu'$. But then again Proposition \ref{pro:KnPn} together with Remark \ref{rmk:grasscells} shows that $\tau=\tau'$ and $I=I'$.

\end{proof}
\begin{example}
A situation where Theorem \ref{thm:cells} turns out to be very useful is the case of roots $\alpha$ such that $\alpha':=\alpha-\alpha_{q_0}q_0$ is a Schur root and $q_0$ is a sink or source of $Q$. Let $Q'$ be the full subquiver of $Q$ with vertices $Q_0\backslash\{q_0\}$ and let $M\in\rep (Q')$. Then we have $\Hom(S_{q_0},M)=\Hom(M,S_{q_0})=0$ and assuming that $q_0$ is a source, we have have $\Ext(M,S_{q_0})=0$ and $n:=\dim\Ext(S_{q_0},M)=\sum_{a:q_0\to q}\alpha_{q}$. 

If $(M,U_M)$ is a cell of Schurian representation of dimension $\alpha'$, Theorem \ref{thm:cells} gives a  $(\Gr_{\alpha_{q_0}}(k^n)\times U_M)$-parameter family of indecomposables of dimension $\alpha$ and thus a mosaic of indecomposables. 
\end{example}

\section{Tree modules and tree normal forms}\label{sec:trees}
\noindent In this section, we apply the results of the previous section to representations which admit a particularly nice basis.  
Again the quiver $Q$ and field $\kk$ are fixed but arbitrary.
\subsection{Tree modules}
Tree modules are quiver representations whose structure can be encoded by a directed graph.  We use the term ``tree modules'' in the more general sense of Ringel \cite{Ringel:1998gf}, which is less restrictive than the usage elsewhere \cite{Crawley-Boevey:1989rr,MR1090218}.
A \emph{morphism of quivers}, written $\tQ \xto{f} Q$, consists of two maps $f_0\colon \tQ_0 \to Q_0$ and $f_1 \colon \tQ_1 \to Q_1$ satisfying $s(f_1(a))=f_0 (s(a))$ and $t(f_1(a)) = f_0(t(a))$ for all $a \in \tQ_0$ (we use $s,t$ for source and target maps of both $\tQ$ and $Q$ since context makes it clear).

\begin{definition}
Fix a quiver $Q$.  A \emph{quiver labeled by} (or \emph{colored by}) $Q$ is a pair $(\tQ, f)$ consisting of a quiver $\tQ$ and a morphism of quivers $\tQ \xto{f} Q$, called the \emph{structure map}.  For each $x \in Q_0$, we refer to $f^{-1}(x)$ as the set of vertices \emph{labeled by} $x$, and similarly for $a \in Q_1$.
\end{definition}

We will specify $f$ in examples by drawing $\tQ$ and labeling its arrows with names of arrows of $Q$, because the compatibility condition determines the labels of the vertices.   We will assume that for each pair of vertices $y,z$ in $\tQ$, there are never two arrows with the same label from $y$ to $z$.
Since we rarely deal with more than one structure map for the same $\tQ$, we usually omit $f$ from the notation.

We are particularly interested in the case that the underlying graph of $\tQ$ is a tree, meaning that it is connected and has exactly one fewer edge than it has vertices.  In this case we usually use the notation $T$ instead of $\tQ$ and, $Q$ being fixed, say $T$ is a \emph{labeled tree}.

We call a $\kk$-basis $\cB$ of $M \in \rep(Q)$ \emph{homogeneous} if $\cB_q := \cB \cap M_q$ is a basis of $M_q$ for each $q \in Q_0$.
Given $M \in \rep(Q)$ and a homogeneous basis $\cB$ of $M$, we define a quiver $\Gamma^\cB$ labeled by $Q$, with structure map $F^\cB\colon \Gamma^\cB\to Q$, known as the \emph{coefficient quiver} of $M$ with respect to $\cB$. We take $\cB$ as the set of vertices of $\Gamma^\cB$, with each subset $\cB_q$ lying over $q$ (i.e. $F^\cB(\cB_q)=q$).  For each $a \in Q_1$ and $b' \in \cB_{s(a)}$, we take the unique expression 
\begin{equation}
M_a(b') = \sum_{b \in \cB_{t(a)}} c^a_{b, b'} b, \qquad c^a_{b, b'} \in \kk,
\end{equation}
and draw a labeled arrow $b' \xto{a} b$ if and only if $c^a_{b, b'} \neq 0$.

\begin{definition}
An indecomposable representation $M$ of $Q$ is said to be a \emph{tree module} if there exists a homogeneous basis $\cB$ of $M$ such that the underlying graph of the coefficient quiver $\Gamma^\cB$ is a tree.
In this case we refer to $\cB$ as a \emph{tree basis} of $M$.
\end{definition}

We emphasize that a tree module is indecomposable by definition in this paper.
A given module can be a tree module with respect to several different bases, yielding different coefficient quivers.  If we work with a fixed but arbitrary tree basis for a given tree module $M$, we omit the basis from the notation and instead denote the coefficient quiver together with its structure map by $F^M \colon \Gamma^M \to Q$.

An equivalent definition is that an indecomposable representation $M$ is a tree module if and only if it admits a matrix presentation consisting of 1s and 0s, with precisely $\dim_\kk M-1$ nonzero entries (which is the minimum possible for $M$ indecomposable).
So we can think of tree modules as indecomposables which can be presented as \emph{sparsely} as possible.

\subsection{Tree-shaped extensions}\label{sec:treeshapedext}

In this section we consider extensions of a nice form with respect to a given basis, for example if the representations in question are tree modules.  We present some tools for recursively constructing such extensions. These tools will help later in various applications, for instance when constructing tree normal forms for quiver representations.

Let $M, N\in \rep(Q)$ be representations with homogeneous bases $\cB_M, \cB_N$.  These bases induce a standard basis of $R(N,M)$ whose elements $f=(f_a)_{a\in Q_1}$ are matrix tuples with exactly one entry of one matrix equal to 1, and the rest equal to 0 which we call standard basis vectors (with respect to $\cB_M$ and $\cB_N$) in the following.

\begin{definition}\label{def:treeshaped}
Let $M, N\in \rep(Q)$ be representations with homogeneous bases $\cB_M, \cB_N$.
 A subset $\cR_{N,M} \subset R(N,M)$ of standard basis vectors \emph{represents a tree-shaped basis} of $\Ext(N,M)$ if it represents a basis of $\Ext(N,M)$ in the sense of Definition \ref{def:representssubset}.
We call the corresponding basis $\cE_{N,M}:=\pi_{N,M}(\cR_{N,M})$ of $\Ext(N,M)$ a \emph{tree-shaped basis of} $\Ext(N,M)$.
\end{definition}
The following is straightforward:
\begin{lemma}\label{lem:treeextquiver}
Let $M, N\in \rep(Q)$ be representations with homogeneous bases $\cB_M, \cB_N$, and $f \in R(N,M)$ be a standard basis vector such that $f_{a_0}(q)=q'$ where $q\in(\cB_N)_{s(a_0)}$, $q'\in(\cB_M)_{t(a_0)}$. Then the coefficient quiver of $B(f) \in \rep(Q)$ with respect to $\cB_N \coprod \cB_M$ is obtained by adding a labeled arrow $q\xrightarrow{a_0} q'$ to $\Gamma^{\cB_N}\coprod\Gamma^{\cB_M}$.
\end{lemma}
In the following, we will often denote the tree-shaped basis element $f\in R(N,M)$ as in the Lemma just by $q\xrightarrow{a_0} q'$. 

 Let $M,N\in\rep(Q)$. For a fixed a short exact sequence $\ses{M}{B}{N}$ and given an additional representation $L$, we can apply the functor $\Hom(L,-)$ to obtain a long exact sequence

\begin{equation}
  \begin{tikzcd}[baseline=(current  bounding  box.center)]
        0 \arrow[r] & \Hom(L,M)\arrow[draw=none]{d}[name=P, shape=coordinate]{} \arrow[r]&\Hom(L,B)\ar[r]& \Hom(L,N)\ar[lld,rounded corners,
to path={ -- ([xshift=5ex]\tikztostart.east)
|- (P) [near end]\tikztonodes
-| ([xshift=-5ex]\tikztotarget.west)
-- (\tikztotarget)},"\delta_L"]&&\\ &\Ext(L,M) \arrow[r]&\Ext(L,B)\arrow[r]& \Ext(L,N) \arrow[r] &0 
    \end{tikzcd}%
\end{equation}   

\begin{lemma}\label{lem:connectingmap}
Fix a short exact sequence $\ses{M}{B}{N}$, where $B$ is given by a tuple $f=(f_a)_{a\in Q_1} \in R(N,M)$.  For an arbitrary $L \in \rep(Q)$, the connecting homomorphism $\delta_{L} \colon \Hom(L,N) \to \Ext(L,M)$ is given by the composition
\begin{equation*}
\Hom(L,N) \to R(L,M) \xto{\pi_{L,M}} \Ext(L,M)
\end{equation*}
where the first map is postcomposition with $f$, that is, $(g_q)_{q\in Q_0}\mapsto (f_a\circ g_{s(a)})_{a\in Q_1}$.

A completely analogous description is obtained for the connecting homomorphism induced by the functor $\Hom(-,L)$, where precomposition replaces postcomposition.
\end{lemma}
\begin{proof}
The connecting homomorphism $\delta_{L}$ sends $g\in\Hom(L,N)$ to the extension of $L$ by $M$ obtained by pulling back along $g$, represented by the commutative diagram

\begin{equation}\xymatrix{0\ar[r]&M\ar[r]&B\ar^{\pi_N}[r]&N\ar[r]&0\\0\ar[r]&M\ar[r]\ar@{=}[u]&C\ar^{\pi_L}[r]\ar[u]^u&L\ar[r]\ar^{g}[u]&0}.\end{equation}
If we write $u=\begin{pmatrix}u_{1,1}&u_{1,2}\\u_{2,1}&u_{2,2}\end{pmatrix}$ with obvious linear maps $u_{i,j}$ and if we write the linear maps $C_a$ as \begin{equation}C_a=\begin{pmatrix}M_a &h_a\\0&L_a\end{pmatrix}\end{equation}
for some $h\in R(L,M)$,
the commutativity of the right square yields $u_{2,1}=0$ and $u_{2,2}=g$. The description of the pullback $C$ as a submodule of $B\oplus L$ gives $u_{1,1}=\id_M$ and $0=u_{1,2}:L\to M$. As furthermore $u:C\to B$ is a morphism of quiver representations, we obtain
\begin{equation}u_{t(a)}\circ C_a=\begin{pmatrix}M_a&h_a\\0&g_{t(a)}\circ L_a\end{pmatrix}=\begin{pmatrix}M_a&f_a\circ g_{s(a)}\\0&N_a\circ g_{s(a)}\end{pmatrix}=B_a\circ u_{s(a)}\end{equation}
for every $a\in Q_1$ which yields the claim.
\end{proof}

Now we see how bases of various $\Ext$-spaces associated to $M, N$ can be used to obtain a basis of $\Ext(B, B)$.

\begin{lemma}\label{tree-shaped-basis} 
Let $M$ and $N$ be two representations and let $\cR_{X,Y}\subset R(X,Y)$ represent bases of $\Ext(X,Y)$ for $X,Y\in\{M,N\}$, in the sense of Definition \ref{def:representssubset}.  Write $\cR_X:=\cR_{X,X}$ for short.

For each $e\in R(N,M)$ and corresponding $\pi_{N,M}(e):\ses{M}{B}{N}$,
 there exist subsets $\cR_M'\subset\cR_M$, $\cR_N'\subset \cR_N$ and $\cR_{N,M}'\subset\cR_{N,M}$ 
such that, under the identification \eqref{eq:RBB}, $\cR_M'\cup\cR_N'\cup\cR_{N,M}'\cup\cR_{M,N} \subset R(B,B)$ represents a basis of $\Ext(B,B)$.
\end{lemma}
\begin{proof}
The induced long exact sequence
\begin{equation}
\Hom(B, N) \xto{\delta_B} \Ext(B, M) \to \Ext(B, B) \to \Ext(B, N) \to 0
\end{equation}
identifies $\Ext(B,B) \cong \Ext(B, M)/\im \delta_B\oplus \Ext(B, N)$.  We proceed by finding a basis of each direct summand.
 
The exact sequence
\begin{equation}
\Hom(M,N) \xto{\delta_N} \Ext(N,N) \to \Ext(B,N) \to \Ext(M,N) \to 0
\end{equation}
further decomposes the second summand as
\begin{equation}
\Ext(B, N)\cong  \Ext(N,N)/\im\delta_N\oplus\Ext(M,N).
\end{equation}
Since the residue classes of $\pi_{N,N}(\cR_N) \cup \pi_{M,N}(\cR_{M,N})$ span the right hand side, we can choose a subset $\cR'_N \cup \cR_{M,N} \subseteq \cR_N \cup \cR_{M,N}$ which represents a basis of $\Ext(B, N)$.

Similarly, we can use the exact sequence
\begin{equation}
\Hom(M,M) \xto{\delta_M} \Ext(N,M) \to \Ext(B,M) \to \Ext(M,M) \to 0
\end{equation}
to decompose $\Ext(B,M)$ and find 
\begin{equation}\label{eq:ExtBM}
\Ext(B, M)/\im \delta_B \cong \left(\Ext(N,M)/\im \delta_M\oplus\Ext(M,M)\right)/\im\delta_B .
\end{equation}
As above, we first choose a subset $\cR_{N,M}''\cup\cR_M \subseteq \cR_{N,M}\cup\cR_M$ representing a basis of the space  $\Ext(N,M)/\im\delta_M\oplus \Ext(M,M)$, then in a second step, we choose $\cR_{N,M}'\cup\cR_M' \subseteq \cR_{N,M}''\cup\cR_M$ representing a basis of the quotient on the right hand side of \eqref{eq:ExtBM}.
Combining this with the previous paragraph, we get a subset of $R(B,B)$ representing a basis of $\Ext(B,B)$.
\end{proof}

\begin{corollary}\label{cor:extbasis}
Let $M$ and $N$ be Schurian representations with $\Hom(M,N)=0$ and let $\cR_{X,Y}\subset R(X,Y)$ represent bases of $\Ext(X,Y)$ for $X,Y\in\{M,N\}$. If $e\in\cR_{N,M}$, the set 
\begin{equation*}\cR_{B}=\cR_{N,M}\backslash\{e\}\cup\cR_M\cup\cR_N\cup \cR_{M,N}\end{equation*}
represents a basis of $\Ext(B,B)$.
\end{corollary}
\begin{proof}We go through the proof of Lemma \ref{tree-shaped-basis} in this special case. As we have $\End(N)=\kk$ and $\Hom(M,N)=0$, it follows that $\Hom(B,N)=\kk$. Thus $\Hom(B,B)\to \Hom(B,N)$ is surjective because the image of $\id_B$ is $\pi$. Thus $\Ext(B,B)\cong \Ext(B,M)\oplus \Ext(B,N)$. Again $\Hom(M,N)=0$ yields $\Ext(B,N)\cong\Ext(N,N)\oplus\Ext(M,N)$. 

Finally, we have $\delta_M(\mathrm{id}_M)=\pi_{N,M}(e)$ by the standard interpretation of the connecting homomorphism (see Lemma \ref{lem:connectingmap}). Then $\End(M)=\kk$ gives that $\im \delta_M = \langle \pi_{N,M}(e) \rangle$. Thus $\Ext(B,M)\cong \Ext(N,M)/\langle\pi_{N,M}(e)\rangle\oplus\Ext(M,M)$ yields the claim.
\end{proof}

\begin{example} 

We give an example to illustrate the definitions and to show how Lemma \ref{tree-shaped-basis} can be applied to construct tree-shaped bases recursively. 
Denote the arrows for the generalized Kronecker quiver $K(3)$ by $a,b,c$. Consider the tree modules $T_1$ and $T_{2}$ defined by the coefficient quivers
\begin{center}

\begin{tikzpicture}[scale=0.8]

\draw (4,0) node (I1) {$1$} +(2,0) node (J1) {$2$}+(5,0) node (I2) {$1'$} +(7,1) node (J21) {$2'$} +(7,-1) node (J22) {$3'$}; 
\draw[->] (I1)--(J1) node[pos=.5,above] {$a$};
\draw[->] (I2)--(J21) node[pos=.5,above] {$b$};
\draw[->] (I2)--(J22) node[pos=.5,above] {$a$};
\end{tikzpicture}
\\

\end{center}
The following sets represent tree-shaped bases of the respective $\Ext$-spaces:
\begin{eqnarray}
\cR_{T_1}=\{1\xrightarrow{b} 2,\,1\xrightarrow{c} 2\},\,\cR_{T_{2},T_{2}}=\{1'\xrightarrow{c} 2',\,1'\xrightarrow{c}3'\},\nonumber\\
\cR_{T_1,T_{2}}=\{1\xrightarrow{c}2',\,1\xrightarrow{c}3',\,1\xrightarrow{b}3'\},\,\cR_{T_2,T_1}=\{1'\xrightarrow{c}2\}.
\end{eqnarray}
Consider the tree module $T$ defined as the middle term of the short exact sequence $\pi_{T_1,T_2}(1\xrightarrow{c}2')$.  From Lemma \ref{lem:treeextquiver}, its coefficient quiver is below.
\begin{center}
\begin{tikzpicture}[scale=0.8]

\draw (0,0) node (I1) {$1$} +(2,1) node (J1) {$2$}+(0,-2) node (I2) {$1'$} +(2,-1) node (J21) {$2'$} +(2,-3) node (J22) {$3'$}; 
\draw[->] (I1)--(J1) node[pos=.5,above] {$a$};
\draw[->] (I2)--(J21) node[pos=.5,above] {$b$};
\draw[->] (I2)--(J22) node[pos=.5,above] {$a$};

\draw[->] (I1)--(J21) node[pos=.5,above] {$c$};
\end{tikzpicture}
\end{center}
We apply Lemma \ref{tree-shaped-basis} to decompose $\Ext(T,T)$ as
\begin{equation}
\Ext(T_1,T_1)/\im\delta_{T_1}\oplus\Ext(T_2,T_1)\oplus\left(\Ext(T_1,T_2)/\im \delta_{T_2}\oplus\Ext(T_2,T_2)\right)/\im\delta_T, 
\end{equation}
where
\begin{equation}
\Hom(T_2,T_1) \xrightarrow{\delta_{T_1}} \Ext(T_1,T_1),\qquad \Hom(T_2,T_2) \xrightarrow{\delta_{T_2}} \Ext(T_1,T_2),\qquad \Hom(T,T_1) \xrightarrow{\delta_{T}} \Ext(T,T_2)
\end{equation}
are the respective connecting homomorphisms.
A tree-shaped basis of $\Ext(T,T)$ can then be obtained using Lemma \ref{lem:connectingmap} to explicitly determine the images of the connecting maps.
A basis of $\Hom(T_{2},T_1)$ is given by the map $f$ such that $f(1')=1$, $f(3')=2$, and $f$ is zero on the rest of the tree basis.  As $f(2')=0$, we have $\delta_{T_1}(f)=0$ which implies $\im\delta_{T_1}=0\subseteq \Ext(T_1,T_1)$.
As $T_2$ is Schurian, $\delta_{T_2}(\id_{T_2})=\pi_{T_1,T_2}(1\xrightarrow{c}2')$ yields $\im\delta_{T_2}=\langle\pi_{T_1,T_2}(1\xrightarrow{c}2')\rangle$. 
Finally, we have $\im\delta_T\cong\langle \pi_{T,T_2}(1'\xrightarrow{c}2')\rangle\subseteq\Ext(T,T_{2})$.
Therefore, we obtain our subset representing a tree-shaped basis of $\Ext(T,T)$:
\begin{equation}\cR_T=\{1\xrightarrow{b} 2,\,1\xrightarrow{c} 2,\,1'\xrightarrow{c} 2,\,1\xrightarrow{c}3',\,1\xrightarrow{b}3',\,1'\xrightarrow{c}3'\}.\end{equation}
 \end{example}

\subsection{Tree normal forms}
Let $\alpha\in\ZZ_{\geq 0}^{Q_0}$. If we refer to a tree module $T\in R_\alpha(Q)$, we already assume that the number of nonzero entries of the matrix tuple $(T_a)_{a\in Q_1}$ is $\dim_\kk T-1$, i.e. the coefficient quiver of the given matrix form is a tree. Kac's conjecture \cite[Conjecture 3]{Ka83} discussed in Section \ref{intro1} motivates the following definition. 
\begin{definition}\label{def:tnf} Let $\alpha\in\ZZ_{\geq 0}^{Q_0}$ be a root.\begin{enumerate}
\item Let $T\in R_\alpha(Q)$ be a tree module and $U\subseteq R(T,T)$ represent a subset
of $\Ext(T,T)$.  
We say that $M\in R_\alpha(Q)$ has a \emph{$(T,U)$-normal form} 
if there exists a $\lambda\in U$ such that $M\cong T(\lambda)$. 
\item We say that a subset $\cU\subset\mathrm{Ind}(Q,\alpha)$ admits a tree normal form if there exists a collection of tree modules $\{T_i \}_{i=1}^r \subset R_\alpha(Q)$ and subsets $\{U_i \subset R(T_i,T_i)\}_{i=1}^r$, 
with $U_i$ representing a subset of $\Ext(T_i,T_i)$, such that 
every indecomposable representation $M\in\cU$ has a $(T_i, U_i)$-normal form for some $1 \leq i \leq r$.
\item We say that a subset $\cU\subset\mathrm{Ind}(Q,\alpha)$ admits a \emph{cellular tree normal form} if it admits a tree normal form as in $(2)$ such that $\{(T_i,U_i)\mid i=1,\ldots,r\}$ is a mosaic of indecomposable representations.

\item  We say that $\alpha$ admits a (cellular) tree normal form if $\mathrm{Ind}(Q,\alpha)$ admits a (cellular) tree normal form.
\end{enumerate} 
\end{definition}

\begin{example}
This example is well-known but serves to illustrate the definitions.
Consider the Kronecker quiver $K(2)$ with arrows $a, b$. Furthermore, consider the tree module $T$ of $K(2)$ of dimension $(2,2)$ defined by the matrices $T_a=\begin{pmatrix}1&0\\0& 1\end{pmatrix}$ and $T_b=\begin{pmatrix}0&1\\0& 0\end{pmatrix}$ and 
\begin{equation}f=(f_a,f_b)=\left(\begin{pmatrix}0&0\\0& 0\end{pmatrix} ,\begin{pmatrix}1&0\\0& 1\end{pmatrix}\right)\in R(T,T).\end{equation}
Then $U_T=\langle f\rangle$ is strong and separating. Moreover, we have $\Ext(T,T)=\pi_{T,T}(U_T)$, but note that $\{f\}$ is not a tree-shaped basis. 
 
We additionally fix the tree module $S$ defined by $S_a=T_b$ and $S_b=T_a$ and fix the subspace $U_S=\{0\}\subset R(S,S)$. Then $\{(S,U_S),(T,U_T)\}$ is a mosaic of indecomposable representations which gives a cellular tree normal form for the root $(2,2)$.

Note that an analogous decomposition into affine cells is present for the dimension vector $(d,d)$ for $d\geq 2$. Furthermore, this shows that Kac's Conjecture 3 is true in this case as we have $a_{(d,d)}(q)=a_{(2,2)}(q)=q+1$.
\end{example}

This motivates the following conjecture.  It can be seen as a generalization of Kac's Conjecture 3 where a mechanism for constructing the cells is proposed.
We note that it is likely quite difficult.
\begin{conjecture}\label{conj:ctnf}
Let $Q$ be a quiver and $\alpha$ a root for $Q$.  Then $\alpha$ admits a cellular tree normal form, with $c_i$ cells of dimension $i$ in the notation of Section \ref{intro1}.
\end{conjecture}

Often it is possible to construct mosaics of indecomposables or even (cellular) tree normal forms recursively. In order to give an idea how to apply the results of Section \ref{sec:homological} in the setup of tree modules, we restrict to one special case which
is the $d=1$ case of Theorem \ref{thm:cells} for tree modules. 

\begin{theorem}\label{thm:treecells}
Let $S$ and $T$ be tree modules and $(S,U_S)$ and $(T,U_T)$ be cells of Schurian representations.  Suppose that $\Hom(T(\mu),S(\lambda))=0$ for all $(\lambda, \mu)\in U_S\times U_T$ and assume that $\cR_{S,T}=\{e_1,\ldots,e_n\}$ is a universal tree-shaped basis for $(U_T,U_S)$. Then there exist tree modules $B_1,\ldots,B_n$, which are the middle terms of the short exact sequences $\pi_{S,T}(e_i)$, and affine spaces $A_i\subset R(B_i,B_i)$ of dimension $i-1$ such that $\{(B_i,A_i)\mid i=1,\ldots,n\}$ is a mosaic of indecomposables. In particular, each representation $B_i(\nu)$ with $\nu\in A_i$ has a $(B_i,A_i)$-normal form.
\end{theorem}



\begin{remark}
A tree module $T$ is clearly defined over $\Z$.
In this remark let us consider the situation where $R(T,T)_{\ZZ}$ is the corresponding product of matrix spaces over $\ZZ$, and we have $U_\ZZ \subseteq R(T,T)_\ZZ$ defined over $\ZZ$. For example, if $U$ is spanned by some standard basis vectors of the matrix space, the family $T(U_\ZZ):=\setst{T(\zl)}{\zl \in U_\ZZ}$ is represented in matrix form by replacing some 0s with $*$s of variable entries in the matrix form of $T$.  We can then base change to any field $\kk$ and ask whether the resulting space $U_\kk$ is strong and separating, or what is the largest subset which has these properties.  It would be particularly interesting to compare the results for $U_\CC$ and $U_{\F_q}$, but this seems to be a hard problem.\end{remark}

\subsection{Subspace quiver: an example}\label{sec:subspace}
We consider the $n$-subspace quiver $S(n)$ with vertices $q_0,q_1,\ldots,q_n$ and  arrows $a_i:q_i\to q_0$ for $i=1,\ldots,n$.
We consider the root $\alpha(n)=2q_0+\sum_{i=1}^nq_i$. Write $a_n(q):=a_{\alpha(n)}(q)$ for the Kac polynomial.  
\begin{theorem} Let $n\geq 3$. The root $\alpha(n)$ admits a cellular tree normal form over any field $\kk$. Moreover, $a_3(q)=1$ and for $n\geq 4$, we have
\begin{equation}a_n(q)=(q+1)a_{n-1}(q)+2^{n-2}-1.\end{equation}
\end{theorem}
\begin{proof}
We first give a cellular tree normal form for $\alpha(n)$ by induction on $n$, then compute the Kac polynomial by specializing to $\kk=\F_q$ and using the dimensions of the cells.
If $n=3$, the root $\alpha(3)$ is exceptional which means that there exists precisely one indecomposable $T_1^3$ up to isomorphism.  It is given by the matrices
\begin{equation}((T^3_1)_{a_1},(T^3_1)_{a_2},(T^3_1)_{a_3})=\left(\begin{pmatrix}1\\0\end{pmatrix},\begin{pmatrix}0\\1\end{pmatrix},\begin{pmatrix}1\\1\end{pmatrix}\right).\end{equation}
Thus $a_{3}(q)=1$ and $\alpha(3)$ admits a cellular tree normal form with a single cell of dimension 0. 

Let us assume that we constructed a cellular tree normal form for $\alpha(n)$. Thus there exist tree modules $T^n_1,\ldots,T^n_{a_n(1)}$ and affine subspaces $U^n_i\subset R(T^n_i,T^n_i)$ representing subspaces of $\Ext(T^n_i,T^n_i)$ such that $\{(T^n_i,U^n_i)\mid i=1,\ldots, a_n(1)\}$ is a mosaic of indecomposable representations giving a cellular tree normal form for $\alpha(n)$.

Now let $S_{n+1}$ be the simple representation corresponding to $q_{n+1}$.
For all $i$, we have that
\begin{equation}\cR= \left\{e_1 =\begin{pmatrix}1\\0\end{pmatrix},e_2=\begin{pmatrix}0\\1\end{pmatrix}\right\} \subset \Hom_{\kk}((S_{n+1})_{q_{n+1}}, (T_i^n)_{q_0})= R(S_{n+1},T^n_{i}) \end{equation}
represents a tree-shaped basis of $\Ext(S_{n+1},T^n_{i})$. This basis is clearly universal for $(\{0\}, U^n_i)$.  
Let $T_{i,1}^{n+1}$ and $T_{i,2}^{n+1}$ be the middle terms of the short exact sequences represented by $e_1, e_2\in R(S_{n+1},T_i^n)$ respectively.

It is straightforward to check that the hypotheses of Theorem \ref{thm:strong} and \ref{thm:separating} hold where we take $M=T^n_i$ and $N=S_{n+1}$ (the $\Theta$ maps are all zero since $T^n_i$ and $S_{n+1}$ have disjoint support).

This yields two strong and separating cells $U_{i,1}^{n+1}=\{0\} \times U_i^n\subset R(T_{i,1}^n,T_{i,1}^n)$ and $U_{i,2}^{n+1}=\langle e_1\rangle \times U_i^n\cong U_i^n\subset R(T_{i,2}^n,T_{i,2}^n)$, see also Theorem \ref{thm:cells}. This gives two cells of indecomposables whose points have a $(T_{i,j}^{n+1},U_{i,j}^{n+1})$-normal form for $j=1,2$. By induction assumption, every indecomposable representation of $S(n+1)$ of dimension vector $\alpha(n+1)$ which restricts to an indecomposable of $S(n)$ arises in this way.

In the case $\kk=\F_q$, as $\alpha(n)$ is coprime, the absolutely indecomposable representations coincide with the indecomposable representations. This follows from \cite[Section 1.14]{Ka83} as an indecomposable representation over $\F_q Q$ which is not absolutely indecomposable decomposes into a direct sum of absolutely indecomposable representations with the same dimension vector over $\overline{\F_q} Q$.
Therefore our considerations show that there exist $(q+1)a_{n}(q)$ absolutely indecomposable representations of dimension $\alpha(n+1)$ over $\F_q$ which restrict to an indecomposable representation of dimension $\alpha(n)$, in other words those kind of representations contribute $(q+1)a_{n}(q)$ to the Kac polynomial $a_{n+1}(q)$.  

For the remaining indecomposable representations of dimension $\alpha(n+1)$ of $S(n+1)$, the restriction $M|_{S(n)}$ to $S(n)$ is decomposable. Let $M$ be such an indecomposable. As at least three of the subspaces $M_{a_i}(M_{q_i})\subset M_{q_0}$ for $i=1,\ldots, n+1$ need to be different, there exists a nontrivial partition $I\coprod J=\{1,\ldots,n\}$ such that $M|_{S(n)}$ has the following matrix presentation (up to isomorphism)
\begin{equation}M_{a_{i}}=\begin{pmatrix}1\\0\end{pmatrix}\text{ for all }i\in I,\,M_{a_{j}}=\begin{pmatrix}0\\1\end{pmatrix}\text{ for all }j\in J.\end{equation}
As $M$ is indecomposable, this gives (again up to isomorphism) \begin{equation}M_{a_{n+1}}=\begin{pmatrix}1\\1\end{pmatrix}.\end{equation}
As the number of such partitions is $2^{n-1}-1$ and as every such indecomposable representation gives an affine cell of dimension zero, this gives the contribution $2^{n-1}-1$ to $a_{n+1}(q)$, completing the induction for this formula. 
Moreover, the coefficient quiver of $M$ in this basis is a tree.  Thus we obtain a cellular tree normal form for $\alpha(n+1)$, completing the induction for that claim.
\end{proof}
The case $n=4$ is treated in Example \ref{ex:subspace n4}.
We also note that computation of the Kac polynomial for this example is considered as an example of general methods unrelated to ours in \cite{GLRV}. Cell decompositions and normal forms are not considered there.



\section{Construction of families of indecomposables via geometric methods}\label{sec:geometric}
\noindent In this section, we assume that $Q$ is an acyclic quiver and $\kk=\CC$. We consider moduli spaces of stable representations together with a torus action. The resulting Biay{\l}nicki-Birula decomposition can be used to associate an affine space in the moduli space with every torus fixed point. We lift this cell to the representation variety, which then can be understood as a subspace of deformations of the torus fixed point. We show that a subgroup of the general linear group acts on the lifted attracting cell. As this action is much easier to handle as the action of the general linear group, this can often be used to construct a cell of stable representations around each torus fixed point.

\subsection{Moduli spaces}\label{sec:moduli}
For an introduction to the theory of moduli spaces of quiver representations we refer to \cite{Kin94,Reineke:2008fk}.
 We choose a vector $\Theta\in\Z^{Q_0}$ and define a linear form $\Theta\in\mathrm{Hom}(\Z ^{Q_0},\Z)$ by $\Theta(\alpha)=\sum_{q\in Q_0}\Theta_q\alpha_q$. This gives rise to a slope function $\mu:\ZZ_{\geq 0}^{Q_0}\backslash\{0\}\rightarrow\Q$ by \begin{equation}\mu(\alpha)=\frac{\Theta(\alpha)}{\dim(\alpha)}\end{equation}
where $\dim(\alpha)=\sum_{q\in Q_0}\alpha_q$. For a representation $M$ of the quiver $Q$, we define $\mu(M):=\mu(\udim M)$. The representation $M$ is called (semi-)stable if the slope (weakly) decreases on proper nonzero subrepresentations. For a fixed slope function as above, we denote by $R^{\Theta-\mathrm{sst}}_\alpha(Q)$ the set of semistable points and by $R^{\Theta-\mathrm{st}}_\alpha(Q)$ the set of stable points in $R_\alpha(Q)$. Following \cite{Kin94}, there exist moduli spaces $\mo$ (resp. $M^{\Theta-\mathrm{sst}}_\alpha(Q)$) of stable (resp. semistable) representations parametrizing isomorphism classes of stable (resp. polystable) representations. If $Q$ is acyclic and $M^{\Theta-\mathrm{st}}_\alpha(Q)\neq \emptyset$, it is a smooth irreducible variety of dimension $1-\langle \alpha,\alpha\rangle$. Moreover, it is projective if semistability and stability coincide. Recall that this is the case if $\alpha$ is $\Theta$-coprime, i.e. if we have $\mu(\beta)\neq\mu(\alpha)$ for all dimension vectors $0\neq\beta<\alpha$. 

In the following, we denote the quotient morphism by $\pi^\Theta_\alpha:R^{\Theta-\mathrm{st}}_\alpha(Q)\to\mo$ or just by $\pi$ if we fixed a dimension vector and a stability. 

\subsection{Universal abelian covering quiver}\label{univcover}
For an introduction to covering theory we refer to \cite{Gab81,Green83}. Let $A_Q$ be the free abelian group generated by $Q_1$, writing $e_a$ for the basis vector of $A_Q$ corresponding to an arrow $a\in Q_1$.


\begin{definition}
The \emph{universal abelian covering quiver} $\hat Q$ of $Q$ has vertex set $\hat Q_0=Q_0\times A_{Q}$ and arrow set $\hat Q_1=Q_1\times A_{Q}$.  The source and target of an arrow in $\hat Q$ are $(s(a),\chi)\xrightarrow{(a,\chi)} (t(a),\chi+e_a)$.

We say that a representation $M\in \rep(Q)$ \emph{can be lifted to $\hat Q$} if there exists a representation $\hat M\in\rep(\hat Q)$ such that $F_Q\hat M=M$ where $F_Q$ is the natural pushdown functor.
\end{definition}

Note that in our definition every connected component of $\hat Q$ is a covering in the sense of \cite{Gab81}. The functor $F_Q$ induces a map $F_Q:\ZZ_{\geq 0}^{\hat Q_0}\to \ZZ_{\geq 0}^{Q_0}$. We say that a dimension vector $\hat\alpha$ is \emph{compatible} with $\alpha$ if $F_Q(\hat\alpha)=\alpha$. The group $A_Q$ acts on $\hat Q$ via translation inducing an action on $\rep(\hat Q)$ and on $\ZZ_{\geq 0}^{\hat Q_0}$. We say two representations are \emph{equivalent} if they lie in the same orbit under this action. If $M$ is a representation of $\hat Q$, we denote the representation obtained by translation by $\chi\in A_Q$ by $M_\chi$. 
The following is straightforward:

\begin{lemma}
Every tree module can be lifted to the universal abelian covering quiver.\end{lemma}

The following result uses that $F_Q$ is a covering functor when restricting to one of the connected components of $\hat Q$.
\begin{theorem}\label{covering}
The functor $F_Q$ preserves indecomposability. Moreover, for all representations $\hat M,\hat N \in\rep(\hat Q)$, we have 
\begin{equation}\Hom_Q(F_Q\hat M, F_Q\hat N)\cong \bigoplus_{\chi\in A_Q}\Hom_{\hat Q}(\hat M_\chi,\hat N)\cong\bigoplus_{\chi\in A_Q}\Hom_{\hat Q}(\hat M,\hat N_\chi).\end{equation}
The analogous statement is true when replacing $\Hom$ by $\Ext$.
\end{theorem} 


\subsection{Torus action on moduli spaces}\label{sec:torusaction}

Let the torus $\T:=(\C^\ast)^{|Q_1|}$ act on $R_\alpha(Q)$ by
\begin{equation}t.M=(t_a)_{a\in Q_1}.(M_a)_{a\in Q_1}:=(t_aM_a)_{a\in Q_1}.\end{equation}
This action commutes with the base change action of $\Gl_\alpha:=\prod_{q\in Q_0}\mathrm{Gl}_{\alpha_q}(\C)$  on $R_\alpha(Q)$ given by \begin{equation}g\ast M:=(g_{t(a)}M_ag_{s(a)}^{-1})_{a\in Q_1}.\end{equation} As the $\T$-action preserves the submodule lattice, it also preserves stability, so this induces a $\T$-action on the moduli space $M^{\Theta-\mathrm{st}}_\alpha(Q)$.
 
We recall some results from \cite[Section 3]{Wei13} which are important for our purposes. Let $\mathrm{PGl}_\alpha=\mathrm{Gl}_\alpha/\C^\ast$, where $\C^* \vartriangleleft \Gl_\alpha$ is the normal subgroup $\{(\lambda\id_{\alpha_q})_{q \in Q_0} \mid \lambda \in \C^*\}$. For every $\T$-fixed point $\overline T\in\mo$, we can choose a representative - also called a \emph{lift} in what follows - $T\in R^{\Theta-\mathrm{st}}_\alpha(Q)$. Every such lift gives rise to a unique homomorphism of algebraic groups $\phi:\T\to \mathrm{PGl}_\alpha$ such that 
\begin{equation}\phi(t)\ast T=t.T.\end{equation}
 For $\phi$ we can choose a lift $\psi:\T\to \Gl_\alpha$ which is unique up to a character $\chi:\T\to\C^\ast$. Every such lift $\psi$ can be decomposed as $\psi=(\psi_q)_{q\in Q_0}$ and gives rise to a weight space decomposition
\begin{equation}T_q=\bigoplus_{\chi\in X(\T)} (T_q)_\chi\end{equation}
for every $q\in Q_0$. Here $X(\T)\cong\Z^{Q_1}$ denotes the character group. Furthermore, we have $T_a(T_{s(a),\chi})\subseteq T_{t(a),\chi+e_a}$ for each $a\in Q_1$. Thus, $T$ defines a $\hat\Theta$-stable representation of the universal abelian covering quiver $\hat Q$ as defined in Section \ref{univcover}. Here the linear form $\hat\Theta\in\Hom(\Z^{Q_0},\Z)$ is defined by $\hat\Theta_{(q,\chi)}=\Theta_q$ for all $q\in Q_0,\,\chi\in A_Q$. Note that a change of the lift $\psi$ by $\chi$ corresponds to a translation of the representation in the universal abelian covering quiver.

The other way around, every $\hat\Theta$-stable representation $T\in R_{\hat\alpha}(\hat Q)$ defines a torus fixed point of $\mo$ if $\hat\alpha$ is compatible with $\alpha$. Following \cite[Section 3.2]{Wei13}, for $\psi_T:\T\to\Gl_\alpha$ defined by
\begin{equation}\label{eq:normalform}(\psi_T)_q(t)(x_{(q,\chi)})=\chi(t)x_{(q,\chi)}\end{equation}
for each $t\in\T$ and $x_{(q,\chi)}\in T_{(q,\chi)}$, we have $\psi_T(t)\ast T=t.T$. Thus $T$ is indeed a fixed point. In \cite[Theorem 3.8]{Wei13} it is shown:

\begin{theorem}
The set of torus fixed points $M^{\Theta-\mathrm{st}}_\alpha(Q)^\T$ is isomorphic to the disjoint union of moduli spaces\
\begin{equation*}\coprod_{\hat \alpha}M^{\hat\Theta-\mathrm{st}}_{\hat{\alpha}}(\hat Q)\end{equation*}
where $\hat\alpha$ ranges over all equivalence classes of dimension vectors compatible with $\alpha$. 
\end{theorem}


\subsection{Bia{\l}ynicki-Birula decomposition for moduli spaces}\label{birula}
We fix the following assumption for the remainder of this section.
\begin{assumption}
Assume that $\alpha\in\ZZ_{\geq 0}^{Q_0}$ is $\Theta$-coprime so that the moduli space $\mo$ is smooth and projective, as discussed in Section \ref{sec:moduli}.
\end{assumption}

 Let $Z$ be a smooth projective variety with a $\C^\ast$-action. For a connected component of the fixed point set $C\subset Z^{\C^\ast}$, we define its attracting set as
\begin{equation}\mathrm{Att}(C):=\{y\in Z\mid \lim_{t\to 0}t.y\in C\}\end{equation}

Then we have the following statement \cite[Section 4]{BB73}:
\begin{theorem}\label{bb}Let $\coprod_{i=1}^r C_i= Z^{\C^\ast}$ be the decomposition into connected components. Then $\mathrm{Att}(C_i)$ is a locally closed smooth $\C^\ast$-invariant subvariety of $Z$ whence $C_i$ is a subvariety of $\mathrm{Att}(C_i)$. Moreover, we have $Z=\coprod_{i=1}^r\mathrm{Att}(C_i)$ and
the natural map $\gamma_i:\mathrm{Att}(C_i)\to C_i$ is an affine bundle.

\end{theorem}

In order to apply Theorem \ref{bb} to the torus action defined in Section \ref{sec:torusaction}, we can define a $\C^\ast$-action on $M^{\Theta-\mathrm{st}}_{ \alpha}( Q)$ with the same fixed point set. 
Once we do this, it follows directly that the moduli space of stable representations admits a cell decomposition into affine spaces if the fixed point set is finite.
To do so, we fix a one-parameter subgroup $\gamma=(\gamma_a)_{a\in Q_1}:\C^\ast\to (\C^\ast)^{\mid Q_1\mid}$ which is sufficiently general and consider the induced $\C^\ast$-action on $R^{\Theta-\mathrm{st}}_\alpha(Q)$, i.e.
\begin{equation}t.(M_a)_{a\in Q_1}:=(\gamma_a(t)M_a)_{a\in Q_1}.\end{equation}
Recall that such a one-parameter subgroup is given by a vector $(\gamma_a)_a\in\Z^{Q_1}$. In \cite[Chapter 2.4]{Pet07}, it is worked out how the attractor sets can be determined for a torus action on a geometric quotient coming from an action of a linear algebraic group on a vector space. We transfer and extend the results to adjust them to our situation including the proofs for completeness. 

Thereby, our main interest is in lifting the attracting set
 \begin{equation}\mathrm{Att}(\bar T)=\{\bar M\in\mo\mid\lim_{t\to 0} t.\bar M=\bar T\}\end{equation}
of a torus fixed point $\bar T\in\mo$ to $\ros$, i.e. we investigate the sets
\begin{equation}\att(T)=\{M\in\ros\mid \lim_{t\to 0}(\psi_T(t),t).M=T\}\end{equation}
for a lift $T\in\ros$. Then the next step is to deduce cells $(T,U)$ of indecomposable representations from $\att(T)$ where $U$ is in bijection with $\att(\bar T)$. If $T$ is a tree module - which is for instance the case if it is exceptional as a representation of $\hat Q$ - this gives a $(T,U)$-normal form for the lifted representations.

\begin{remark}\label{rem:standardformsubgroup}
 With a tree module $T\in R_\alpha(Q)$ with homogeneous basis $\mathcal B_T$, we can associate a subquiver and a dimension vector $\hat \alpha_T$  of the universal abelian covering quiver. Both are unique up to translation by $\chi\in A_Q$. In this way, we can associate a vertex $(q,\chi)$ with every $b\in\cB_T$.

Consider the group homomorphism $d_\gamma:A_Q\to\Z$ by $d_\gamma(e_a)=\gamma_a$.
If $T$ is stable, i.e. $T$ is a torus fixed point, \eqref{eq:normalform} shows that the corresponding one-parameter subgroup $\psi_T:\C^\ast\to\Gl_\alpha$ is given by diagonal matrices with diagonal entries $(\psi_T(t)_{q})_{b,b}=t^{d_\gamma(\chi)}$ where $b\in\cB_T$ is supported at $(q,\chi)$. 
 In particular, $\psi_T$ only depends on the dimension vector $\hat\alpha_T\in\ZZ_{\geq 0}^{\hat Q_0}$.

In the following, we call a one-parameter subgroup $\psi:\C^\ast\to\Gl_\alpha$ \emph{in standard form} if every $\psi_q$ is given by a diagonal matrix. 
\end{remark}

 We define the group 
$\hat G_\alpha=\Gl_\alpha\times \C^{\ast}.$
It acts on $R_\alpha^{\Theta-\mathrm{st}}(Q)$ via
\begin{equation}(g,t).M=t^{-1}.(g\ast M)=g\ast(t^{-1}.M).\end{equation}

Recall that a one-parameter subgroup $\psi$ of $\Gl_\alpha$ consists of a collection $(\psi_q)_{q\in Q_0}$ of one-parameter subgroups $\psi_q:\C^\ast\to\Gl_{\alpha_q}$. In turn a one-parameter subgroup of $\hat G_\alpha$ is obtained by adding a character $\chi\in X(\C^\ast)$, i.e. we have $\chi(t)=t^n$ for some $n\in\Z$.
The group $\Gl_\alpha$ acts on the set of one-parameter subgroups of $\Gl_\alpha$ via conjugation, i.e. we have
\begin{equation}(g.\psi)(t):=(g_q\psi_q(t)g_q^{-1})_{q\in Q_0}\end{equation}
for $\psi:\C^\ast\to\Gl_\alpha$. This induces an action on the set of one-parameter subgroups of $\hat G_\alpha$ via $g.(\psi(t),t^n):=((g.\psi)(t),t^n)$.


We start by proving some technical results which are needed for lifting the attracting cells. A similar result is proved in \cite[Proposition 2.26]{Pet07}:
\begin{lemma}\label{hilbert}
Let $\bar T\in\mo$ be a torus fixed point and let $\bar M\in\mathrm{Att}(\bar T)$. Moreover, let $T\in \ros$ be a lift of $\bar T$ and $M\in\ros$ be a lift of $\bar M$. Then there exists a one-parameter subgroup $\hat\psi:\C^\ast\to \hat G_\alpha$ such that 
\begin{equation*}\lim_{t\to 0}\hat\psi(t).M\in\hat G_\alpha T=\Gl_\alpha T.\end{equation*}  
\end{lemma}
\begin{proof}
The aim is to apply the Hilbert criteria in the form of \cite[Theorem 2.4]{Kra84}, see also \cite[Theorem 4.2]{Bir71}, which states that, for any closed $\hat G_\alpha$-stable subset of the orbit closure $\overline{\hat G_\alpha M}$, there exists such a one-parameter subgroup. Therefore, we need to show that $\hat G_\alpha T$ is a closed subset of  $\overline{\hat G_\alpha M}$. 

As $\bar T$ is a fixed point under the $\C^\ast$-action, we have $\hat G_\alpha T=\Gl_\alpha T$. 
If $\pi$ is the quotient map for the $\Gl_\alpha$-action, we have 
\begin{equation}\pi(\hat G_\alpha M)=\{t.\bar M\mid t\in\C^\ast\}.\end{equation}
 Thus we have $\bar T=\lim_{t\to 0}t.\bar M\in \overline{\pi(\hat G_\alpha M)}$ which shows
\begin{equation}\hat G_\alpha T=\pi^{-1}(\bar T)\subseteq \pi^{-1}(\overline{\pi(\hat G_\alpha M)})\subseteq \pi^{-1}(\pi(\overline{\hat G_\alpha M)})=\overline{\hat G_\alpha M}\end{equation}
where we use that $\pi(\overline{\hat G_\alpha M})$ is closed because $\pi:R^{\Theta-\mathrm{st}}_\alpha(Q)\to\mo$ is a geometric quotient. As 
 $\hat G_\alpha T=\pi^{-1}(\bar T)$ is closed in $\ros$ and contained in $\overline{\hat G_\alpha M}$, it is also closed in  $\overline{\hat G_\alpha M}$.
\end{proof}
The following lemma explains the compatibility of the different actions of $\Gl_\alpha$.

\begin{lemma}\label{lemmaG} Let $\bar T\in\mo$ be a torus fixed point and let $\bar M\in\mathrm{Att}(\bar T)$. Moreover, let $T$ be a lift of $\bar T$, $M$ be a lift of $\bar M$ and $\hat\psi:\C^\ast\to\hat G_\alpha$ be a one-parameter subgroup such that
\begin{equation*}\lim_{t\to 0}\hat\psi(t).M=T.\end{equation*}
For every $g\in\Gl_\alpha$, we have 
\begin{equation*}\lim_{t\to 0}(g.\hat\psi)(t).(g\ast M)=g\ast T.\end{equation*}

\end{lemma}
\begin{proof}
We assume that $\hat \psi=((\psi_q)_q,\chi)$ with $\chi(t)=t^n$. Let $a\in Q_1$. Then we have                           
\begin{eqnarray}
(\lim_{t\to 0}(g.\hat\psi)(t).(g\ast M))_a&=&\lim_{t\to 0}t^{-n\gamma_a}\cdot g_{t(a)}\psi_{t(a)}(t)g_{t(a)}^{-1}g_{t(a)}M_ag_{s(a)}^{-1}g_{s(a)}\psi^{-1}_{s(a)}(t)g_{s(a)}^{-1}\nonumber\\&=&\lim_{t\to 0}t^{-n\gamma_a}\cdot g_{t(a)}\psi_{t(a)}(t)M_a\psi^{-1}_{s(a)}(t)g_{s(a)}^{-1}\nonumber\\&=&g_{t(a)}\cdot(\lim_{t\to 0}t^{-n\gamma_a}\psi_{t(a)}(t)M_a\psi^{-1}_{s(a)}(t))\cdot g_{s(a)}^{-1}\nonumber\\&=&g_{t(a)}(\lim_{t\to 0}\hat\psi(t).M)_a g_{s(a)}^{-1}\nonumber\\&=&(g\ast T)_a
\end{eqnarray}
where we use that the limit $\lim_{t\to 0}\hat\psi(t).M$ exists.
\end{proof}

\begin{lemma}\label{lemmaN} 
Let $\bar T\in\mo$ be a torus fixed point and let $\bar M\in\mathrm{Att}(\bar T)$. Moreover, let $T$ be a lift of $\bar T$, $M$ be a lift of $\bar M$ and $\hat\psi:\C^\ast\to\hat G_\alpha$ be a one-parameter subgroup. Then $\lim_{t\to 0}\hat\psi(t).M=T$ if and only if $\lim_{t\to 0}\hat\psi^m(t).M=T$ for a nonzero integer $m$.
\end{lemma}
\begin{proof}One direction is obvious. Thus assume that  $\lim_{t\to 0}\hat\psi^m(t).M=T$.
As before, we decompose $\hat\psi$ into one-parameter subgroups $\psi_q:\C^\ast\to\Gl_{\alpha_q}$  and a character $\chi\in X(\T)$ with $\chi(t)=t^n$ for some $n\in\Z$. For every $q\in Q_0$, there exists $g_q\in\Gl_{\alpha_q}$ and $a_{i,q}\in\Z$ for $i=1,\ldots,\alpha_q$ such that
\begin{equation}\psi_q(t)=g_q\cdot\mathrm{diag}(t^{a_{1,q}},\ldots,t^{a_{\alpha_q,q}})\cdot g_q^{-1}.\end{equation}
Let $g=(g_q)_{q\in Q_0}$ and $\nu_q(t):=\mathrm{diag}(t^{a_{1,q}},\ldots,t^{a_{\alpha_q,q}})$, i.e. $\hat\psi=g.(\nu,\chi)$. Combining Lemma \ref{lemmaG} with the assumption, we have 
\begin{equation}\lim_{t\to 0}(\nu^m(t),t^{nm}).(g^{-1}\ast M)=g^{-1}\ast T.\end{equation}
Let $M'=g^{-1}\ast M$ and $T'=g^{-1}\ast T$. For an arrow $a\in Q_1$, we have
\begin{equation} ((\nu^{m}(t),t^{nm}).M')_a=t^{-nm\gamma_a}\mathrm{diag}(t^{a_{1,t(a)}},\ldots,t^{a_{\alpha_{t(a)},t(a)}})^m\cdot M'_a\cdot\mathrm{diag}(t^{-a_{1,s(a)}},\ldots,t^{-a_{\alpha_{s(a)},s(a)}})^m.\end{equation}
Thus we have $((\nu(t)^{m},t^{mn}).M')_a)_{i,j}=t^{m(a_{i,t(a)}-a_{j,s(a)}-n\gamma_a)}(M'_a)_{i,j}$. This shows that the existence of the limit is independent of $m$ and it follows that
\begin{equation}\lim_{t\to 0}(\nu(t),t^{n}).(g^{-1}\ast M)=g^{-1}\ast T.\end{equation}
This gives the claim when applying Lemma \ref{lemmaG} again.
\end{proof}
The considerations of Section \ref{sec:torusaction} show that for every fixed lift $T$ of a torus fixed point $\bar T$, there exists a one-parameter subgroup $\psi_T:\C^\ast\to\Gl_\alpha$, unique up to some $\chi\in X(\T)$, such that $\psi_T(t)\ast T=t.T$. We can use this to adopt the proof of \cite[Proposition 2.27]{Pet07} for our purposes:
\begin{proposition}Fix a lift $T\in\ros$ of a fixed point $\bar T\in\mo$. For every $\bar M\in\mathrm{Att}(\bar T)$ there exists a lift $M$ such that 
\begin{equation*}\lim_{t\to 0}(\psi_T(t),t).M=T.\end{equation*} 
\end{proposition}
\begin{proof}
By \cite[Section 3.1]{Wei13}, 
the one-parameter subgroup corresponding to $g\ast T$ where $g\in\Gl_\alpha$ is $\psi_{g\ast T}=g.\psi_T$. Fix any lift $M$ of $\bar M$. By Lemma \ref{hilbert}, there exist a one-parameter subgroup $\hat\psi=(\psi,\chi):\C^\ast\to \hat G_\alpha$ with $\chi(t)=t^n$ for some integer $n$ and a $g\in \Gl_\alpha$ such that
\begin{equation}\lim_{t\to 0}\hat\psi(t).M= g\ast T.\end{equation}
For each $t_0\in \C^\ast$, the existence of the limit can be used to show
\begin{equation}(\psi(t_0),t_0^n).( g\ast T)=((\psi(t_0),t_0^n).((\lim_{t\to 0}(\psi(t),t^n).M)=(\lim_{t\to 0}(\psi(t_0t),(t_0t)^n)).M)= g\ast T.\end{equation}
As $\psi_{g\ast T}(t)\ast (g\ast T)=t.(g\ast T)$ if and only if $\psi_{g\ast T}(t)^n\ast (g\ast T)=t^n.(g\ast T)$,
the uniqueness of $\psi_{g\ast T}$ gives $\psi=\chi\cdot(\psi_{ g\ast T})^n$ and $\hat\psi(t)=(\chi\cdot\psi_{g\ast T}(t)^n,t^n)$ for some character $\chi\in X(\T)$. As the scalars - and thus $\chi$ - act trivially on $\ros$, we can apply Lemma \ref{lemmaN} to obtain 
\begin{equation}\lim_{t\to 0}(\psi_{g\ast T}(t),t).M=g\ast T.\end{equation}
 
As we have $\psi_{g\ast T}=g.\psi_T$, we can apply Lemma \ref{lemmaG} and get $\lim_{t\to 0}(\psi_T(t),t).( g^{-1}\ast M)=T.$ 
\end{proof}
Thus we have shown that every $\bar M\in\att(\bar T)$ has a representative in the lifted attracting cell 
\begin{equation}\att(T)=\{M\in\ros\mid \lim_{t\to 0}(\psi_T(t),t).M=T\}.\end{equation}

 We will see that we have an action of a subgroup of a parabolic subgroup of $\Gl_\alpha$ on $\mathrm{Att}(T)$ whose orbit space 
 can be identified with $\mathrm{Att}(\bar T)$. For a one-parameter subgroup $\psi:\C^\ast\to\Gl_\alpha$ we define 
\begin{equation}P_\psi=\{g\in\Gl_\alpha\mid \lim_{t\to 0} \psi_q(t)g_q\psi_q(t)^{-1}\text{ exists for all }q\in Q_0\}\end{equation}
and consider the subgroup
\begin{equation}U_\psi=\{g\in\Gl_\alpha\mid\exists\,\mu\in\C^\ast: \lim_{t\to 0} \psi_q(t)g_q\psi_q(t)^{-1}=\mu E_{\alpha_q}\text{ for all }q\in Q_0\}.\end{equation}

\begin{remark}
If $\psi_q(t)=\diag(t^{a_{1,q}},\ldots,t^{a_{\alpha_q,q}})$ for all $q\in Q_0$ - which is the case if the lift of a torus fixed point is given as a representation of $\hat Q$ - it is easy to determine the corresponding subgroup $U_\psi$. More precisely, for $g\in U_\psi$ we then have $(g_q)_{i,i}=\mu$ for all $q\in Q_0$ and for some $\mu\in\C^\ast$. Moreover, for $i\neq j$ we have that 
$(g_q)_{i,j}$ is arbitrary if $a_{i,q}-a_{j,q}>0$ and $(g_q)_{i,j}=0$ if $a_{i,q}-a_{j,q}\leq 0$.


\end{remark}
We need another technical lemma:
\begin{lemma}\label{lemma: uaction}
Let $T$ be a lift of a torus fixed point, let $M\in\att(T)$ and let $g\in U_{\psi_T}$. 
\begin{enumerate}
\item Then we have $g\ast M\in\att(T)$.
\item We have $g\ast T\in\att(T)$ if and only if $g\in U_{\psi_T}$.
\item If $\psi_T$ is in standard form, for all arrows $a\in Q_1$, we have $(M_a)_{i,j}=(T_a)_{i,j}$ whenever $(T_a)_{i,j}\neq 0$. 
\end{enumerate}
\end{lemma}
\begin{proof}Let $\psi=\psi_T$. Let $g\in P_{\psi}$ and $M\in\mathrm{Att}(T)$. Then we have
\begin{eqnarray}\lim_{t\to 0}(\psi(t),t).(g\ast M)&=&\lim_{t\to 0}((\psi_q(t) g_q \psi_q(t)^{-1})_{q\in Q_0},t).(\psi(t)\ast M)\nonumber\\&=&(\lim_{t\to 0}(\psi_q(t) g_q \psi_q(t)^{-1})_{q\in Q_0})\ast (\lim_{t\to 0}(\psi(t),t). M)\nonumber\\&=&(\lim_{t\to 0}(\psi_q(t) g_q \psi_q(t)^{-1})_{q\in Q_0})\ast T
\end{eqnarray}
where the equations hold because the respective limits exist. 

Now the endomorphism ring of $T$ is trivial, which means that we additionally have $g\in U_{\psi}$ if and only if
\begin{equation}\lim_{t\to 0}(\psi(t),t).(g\ast M)=T.\end{equation}
This shows the first claim. 

For the second claim, assume that $g\in\Gl_\alpha$ and consider
\begin{eqnarray}\lim_{t\to 0}(\psi(t),t).(g\ast T)&=&\lim_{t\to 0}(\psi_q(t) g_q \psi_q(t)^{-1})_{q\in Q_0})\ast((\psi(t),t).T)\nonumber\\&=&\lim_{t\to 0}(\psi_q(t) g_q \psi_q(t)^{-1})_{q\in Q_0}\ast T
\end{eqnarray}
 where we use that $T$ is a fixed point. Now the same argument applies.

If $\psi$ is in standard from and $M\in\att(T)$, for $a\in Q_1$, we have \begin{equation}\lim_{t\to 0}((\psi(t),t).M_a)_{i,j}=\lim_{t\to 0}t^{-\gamma_a+a_{i,t(a)}-a_{j,s(a)}}(M_a)_{i,j}=(T_a)_{i,j}.\end{equation}
If $(T_a)_{i,j}\neq 0$, it follows that $-\gamma_a+a_{i,t(a)}-a_{j,s(a)}=0$ and thus $(T_a)_{i,j}=(M_a)_{i,j}$.
\end{proof}

The following result shows that the second part of the lemma holds for arbitrary $M\in\att(T)$. A similar result is proved in \cite[Lemma 2.32]{Pet07}:
\begin{proposition}\label{pro:Uaction}
Let $T$ be a lift of a torus fixed point. Then there exists an action of $U_{\psi_T}$ on $\att(T)$ such that for all $M\in\att(T)$ we have $\Gl_\alpha\ast M\cap \mathrm{Att}(T)=U_{\psi_T}\ast M$. In particular, the affine space $\pi(\mathrm{Att}(T))=\att(\bar T)$ is the orbit space for the $U_{\psi_T}$-action on $\att(T)$ which we sometimes write as $\att(T)/U_\psi$.
\end{proposition}
\begin{proof}Write $\psi=\psi_T$. The existence of the $U_\psi$-action is a consequence of Lemma \ref{lemma: uaction}. It remains to show that $g\in U_\psi$ if $g\ast M\in\att(T)$  and $M\in\att(T)$.
Lemma \ref{lemmaG} implies that we can assume that each $\psi_q$ is in standard form, i.e. we assume that $\psi_q(t)=\diag(t^{a_{1,q}},\ldots,t^{a_{\alpha_q,q}})$ for integers $a_{i,q}$. Then we have
\begin{equation}(((\psi(t),t).(g\ast M))_a)_{k,l}=\sum_{i=1}^{\alpha_{t(a)}}\sum_{j=1}^{\alpha_{s(a)}}t^{-\gamma_a+a_{k,t(a)}-a_{l,s(a)}}(g_{t(a)})_{k,i}(M_a)_{i,j}(g_{s(a)}^{-1})_{j,l}.\end{equation}
As we have $(M_a)_{i,j}=(T_a)_{i,j}$ if $(T_a)_{i,j}\neq 0$, this shows that $\lim_{t\to 0}(\psi(t),t).(g\ast T)$ exists whenever $\lim_{t\to 0}(\psi(t),t).(g\ast M))$ exists.  Indeed, for the limit to exist, the limit of every single summand needs to exist. But now the second part of Lemma \ref{lemma: uaction} shows that $g\in U_\psi$.
\end{proof}

The following result translates the results of this section into the language of Section \ref{sec:homological}.
\begin{theorem}\label{thm:geomcells}Let $\bar T\in\mo$ be a torus fixed point and $T\in R_\alpha(Q)$ be a lift. Then there exists a subspace $V_T\subset R(T,T)$ such that $\att(T)=T+V_T$, i.e. for all $M\in\att(T)$ we have $M\cong T(\lambda)$ with $\lambda\in V_T$.

If we choose $U_T\subset V_T$ such that $|(U_{\psi_T}\ast T(\lambda))\cap \att(T)|=1$ for every $\lambda\in U_T$, then $U_T$ represents a strong and separating subset of $\Ext(T,T)$.
\end{theorem}
\begin{proof}
We first choose the lift $T$ of $\bar T$ such that $\psi_T$ is in standard form. This is  the case if $T$ itself can be lifted to the universal abelian covering quiver, see Section \ref{sec:torusaction}. Thus, by Lemma \ref{lemma: uaction}, we have $(M_a)_{i,j}=(T_a)_{i,j}$ if $(T_a)_{i,j}\neq 0$. Note that we also have $-\gamma_a+a_{i,t(a)}-a_{j,s(a)}=0$ in this case. Moreover, if $(T_a)_{i,j}=0$, we have
\begin{equation}\lim_{t\to 0}((\psi(t),t).M_a)_{i,j}=\lim_{t\to 0}t^{-\gamma_a+a_{i,t(a)}-a_{j,s(a)}}(M_a)_{i,j}=0\end{equation}
whenever $-\gamma_a+a_{i,t(a)}-a_{j,s(a)}> 0$ or $(M_a)_{i,j}=0$. This shows
\begin{eqnarray}\att(T)&=&\{M\in R_\alpha(Q)\mid (M_a)_{i,j}=(T_a)_{i,j} \text{ if }(T_a)_{i,j}\neq 0,\,(M_a)_{i,j}=0\text{ if }a_{i,t(a)}-a_{j,s(a)}<\gamma_a\}\nonumber\\&=&T+\{\lambda\in R(T,T)\mid (\lambda_a)_{i,j}=0\text{ if }a_{i,t(a)}-a_{j,s(a)}\leq \gamma_a\}.\end{eqnarray}
Thus the right hand summand defines a subspace $V_T$ of $R_\alpha(Q)=R(T,T)$ such that $T(\lambda)$ is stable for all $\lambda\in V_T$. 

The second part of the statement is clear since every representation $T(\lambda)$ is stable and thus indecomposable. Moreover, Proposition \ref{pro:Uaction} shows that the orbits with respect to the $U_{\psi_T}$-action are in one-to-one correspondence with the isomorphism classes of representations contained in $\att(T)$.

Finally, for another lift $T'=g\ast T$ with one-parameter subgroup $\psi_{g\ast T}$, by Lemma \ref{lemmaG}, we have that $g\ast V_T$ and $g\ast U_T$ satisfy the conditions.
 \end{proof}
\begin{remark}
Actually, the $U_{\psi_T}$-action is much easier to handle than the $\Gl_\alpha$-action as we can mostly choose representatives of $\att(T)/U_{\psi_T}$ in the lifted affine cell $\att(T)$ in a canonical way, i.e. we can choose $U_T\subset V_T$ as a subspace. In this case $(T,U_T)$ defines a cell of stable, and thus indecomposable, representations.

The results show $U_{\psi_T}$ acts freely on the affine space $\att(T)$\ and that we have $\att(T)/U_{\psi_T}=\pi(\att(T))$ for the orbit space. Furthermore, the fibres of $\pi{\mid_{\att(T)}}:\att(T)\to\pi(\att(T))$ are affine spaces of dimension $\dim U_{\psi_T}$. Nevertheless, it seems to be not clear that this map is an affine bundle. If it were an affine bundle, it would be trivial by the Quillen-Suslin theorem because $\pi(\att(T))$ is also an affine space. In particular, there would exist an affine isomorphism $\varphi:\att(T)\to\pi(\att(T))\times U_{\psi_T}$ such that $\mathrm{pr}_1\circ\varphi=\pi{\mid_{\att(T)}}$. If we choose an affine global section $\sigma:\pi(\att(T))\to\pi(\att(T))\times U_{\psi_T}$, then $\varphi^{-1}\circ\sigma$ is an affine global section of $\pi{\mid_{\att(T)}}$ which means that $\varphi^{-1}\circ\sigma(\pi(\att(T)))$ defines an affine subspace of $\att(T)$ of dimension $\dim U_{\psi_T}$. This gives a strong and separating subspace $U_T\subset R(T,T)$ with $T+U_T=(\varphi^{-1}\circ\sigma)(\pi(\att(T))$ in a natural way.
\end{remark}
If $\hat\alpha$ is an exceptional root of $\hat Q$, we denote the unique indecomposable representation (up to isomorphism) of dimension $\hat\alpha$ by $T_{\hat\alpha}$. Recall that $T_{\hat\alpha}$ is a tree module by \cite{Ringel:1998gf}.

\begin{corollary}\label{cor:geomcells} 
Let $\alpha$ be a root such that $\mo$ is not empty. Assume that every root $\hat\alpha$ of $\hat Q$ which is compatible with $\alpha$ and which satisfies $M^{\hat\Theta-\mathrm{st}}_{\hat\alpha}(\hat Q)\neq \emptyset$ is exceptional. For every such $\hat\alpha$, assume that $\pi{\mid_{\att(T_{\hat \alpha})}}:\att(T_{\hat\alpha})\to\pi(\att(T_{\hat\alpha}))$ is an affine bundle.

Then there exists a mosaic of stable representations $\{(T_{\hat\alpha},U_{T_{\hat\alpha}})\}_{\hat\alpha}$  where $\hat\alpha$ runs through all equivalence classes which are compatible with $\alpha$ and satisfy $M^{\hat\Theta-\mathrm{st}}_{\hat\alpha}(\hat Q)=\{\mathrm{pt}\}$. In particular, every stable representation has a $(T_{\hat\alpha},U_{T_{\hat\alpha}})_{\hat\alpha}$-normal form for some $\hat\alpha$.
\end{corollary}
\begin{proof}
The assumptions assure that there exist only finitely many fixed points which are represented by stable representations $T_{\hat\alpha}$ of $\hat Q$ such that $\hat\alpha$ is compatible with $\alpha$. As $\hat\alpha$ is exceptional, $T_{\hat\alpha}$ is a tree module of $\hat Q$ and thus of $Q$. Then Theorem \ref{bb} together with Theorem \ref{thm:geomcells} gives the claim.
\end{proof}

\begin{example}
We state a first easy example which shows in detail how a lifted attractor cell is obtained starting with a fixed point which is given as a representation of the universal abelian covering quiver. In this case this produces also a cell of stable representations.

We consider $K(3)$, the root $(d,e)=(2,3)$ and the stability induced by $\Theta=(1,0)$. We denote the arrows by $a,b$ and $c$ and consider the torus action induced by choosing $\gamma=(1,3,5)$. Then the following tree module $T\in\rep(\hat Q)$ (with weight space decomposition as indicated) is a torus fixed point:
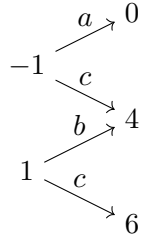
\begin{center}
\begin{tikzpicture}[scale=0.7]

\draw (0,0) node (I1) {$-1$} +(2,1) node (J1) {$0$}+(0,-2) node (I2) {$1$} +(2,-1) node (J21) {$4$} +(2,-3) node (J22) {$6$}; 
\draw[->] (I1)--(J1) node[pos=.5,above] {$a$};
\draw[->] (I2)--(J21) node[pos=.5,above] {$b$};
\draw[->] (I2)--(J22) node[pos=.5,above] {$c$};

\draw[->] (I1)--(J21) node[pos=.5,above] {$c$};
\end{tikzpicture}
\end{center}
Then we have $\psi_T(t)=(\diag(t^{-1},t),\diag(1,t^4,t^6))$.
Furthermore, it is straightforward that we have
\begin{equation}\mathrm{Att}(T)=\left(\begin{pmatrix}1&0\\\ast&\ast\\\ast&\ast\end{pmatrix},\begin{pmatrix}0&0\\\ast&1\\\ast&\ast\end{pmatrix},\begin{pmatrix}0&0\\1&0\\\ast&1\end{pmatrix}\right),\qquad U_{\psi_T}=\left(\begin{pmatrix}1&0\\\ast&1\end{pmatrix},\begin{pmatrix}1&0&0\\\ast&1&0\\\ast&\ast&1\end{pmatrix}\right),\end{equation}
\begin{equation}\mathrm{Att}(T)/U_\psi\cong T+ \left(\begin{pmatrix}0&0\\0&\ast\\0&\ast\end{pmatrix},\begin{pmatrix}0&0\\ 0&0\\\ast&\ast\end{pmatrix},\begin{pmatrix}0&0\\0&0\\ 0&0\end{pmatrix}\right)
\subset R_{(2,3)}(K(3)).\end{equation}
Writing $U_T$ for the right hand summand, $(T,U_T)$ gives a four-dimensional affine cell of stable representations of dimension $(2,3)$.

Actually the moduli spaces $M_{(d,d+1)}^{(1,0)-\mathrm{st}}(K(m))$ all have cell decompositions into affine spaces as there are only finitely many fixed points, see \cite[Section 6.2]{Wei13}. Now the results of this section can be used to obtain a cellular tree normal form for $R_{(d,d+1)}^{(1,0)-\mathrm{st}}(K(m))$, see also \cite[Proposition 7.3]{Reineke:2008fk} where the dimension vector $(2,3)$ for $K(3)$ is treated.  
\end{example}

\subsection{Extended Kronecker quiver: an example}\label{ex:extkronecker}

We consider the quiver
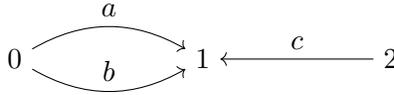
\begin{center}
\begin{tikzpicture}[scale=1]

\draw (0,0) node (I) {$0$} +(2.5,0) node (J) {$1$}+(5,0) node (K) {$2$}; 
\draw[->] (I) edge [bend left=30] node[above] {$a$} (J);
\draw[->] (I) edge [bend right=30] node[above] {$b$} (J);

\draw[->] (K) edge node[above] {$c$} (J);
\end{tikzpicture}
\end{center}
which we denote by $K(2,1)$ in what follows. We consider the stability induced by $\Theta=(1,0,1)$ and the torus action induced by $\gamma=(\gamma_a,\gamma_b,\gamma_c)=(1,3,1)$. Moreover, we consider the dimension vector $(n,n,1)$. The moduli space of stable representations has dimension $1-\langle(n,n),(n,n)\rangle=n$ and $n+1$ torus fixed points 
$T^n_1,\ldots,T^n_{n+1}$ which are defined by the indicated exceptional roots of the following subquivers of the universal abelian cover $\widehat{K(2,1)}$ where the bullets stand for one-dimensional vector spaces - for the purpose of exhibition we display the case $n=4$, the other cases are obtained by the obvious generalization.

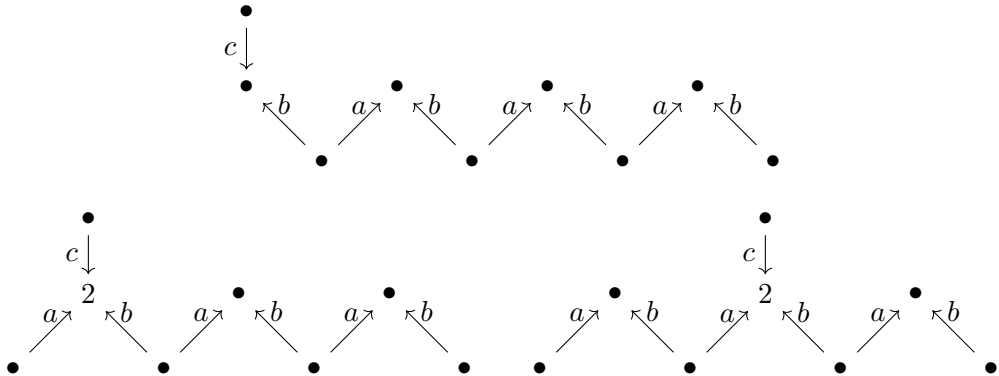
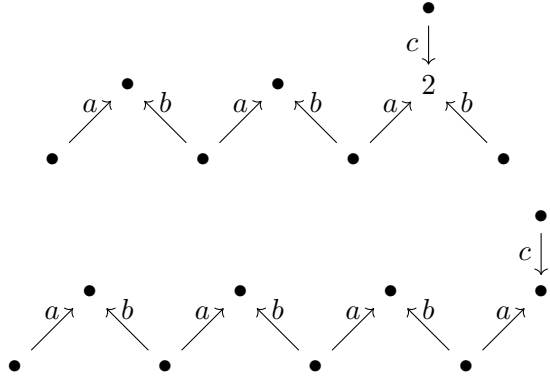
\begin{center}
\begin{tikzpicture}[scale=1]
\draw (0,0) node (11) {$\bullet$} +(2,0) node (12) {$\bullet$}+(4,0) node (13)  {$\bullet$}+(-2,0) node (15) {$\bullet$}+(-1,-1) node (21) {$\bullet$}+(1,-1) node (22) {$\bullet$}+(3,-1) node (23) {$\bullet$}+(5,-1) node (24) {$\bullet$}+(-2,1) node (0) {$\bullet$}; 
\draw[->] (0) edge node[left] {$c$} (15);
\draw[->] (21) edge node[above] {$b$} (15);
\draw[->] (21) edge node[above] {$a$} (11);
\draw[->] (22) edge node[above] {$b$} (11);
\draw[->] (22) edge node[above] {$a$} (12);
\draw[->] (23) edge node[above] {$b$} (12);
\draw[->] (23) edge node[above] {$a$} (13);
\draw[->] (24) edge node[above] {$b$} (13);

\end{tikzpicture}
\\\vs

\begin{tikzpicture}[scale=1]
\draw (0,0) node (11) {$2$} +(2,0) node (12) {$\bullet$}+(4,0) node (13)  {$\bullet$}+(-1,-1) node (21) {$\bullet$}+(1,-1) node (22) {$\bullet$}+(3,-1) node (23) {$\bullet$}+(5,-1) node (23b) {$\bullet$}+(0,1) node (0) {$\bullet$}; 
\draw[->] (0) edge node[left] {$c$} (11);
\draw[->] (21) edge node[above] {$a$} (11);
\draw[->] (22) edge node[above] {$b$} (11);
\draw[->] (22) edge node[above] {$a$} (12);
\draw[->] (23) edge node[above] {$b$} (12);
\draw[->] (23) edge node[above] {$a$} (13);

\draw[->] (23b) edge node[above] {$b$} (13);
\draw (7,0) node (11) {$\bullet$} +(2,0) node (12) {$2$}+(4,0) node (13)  {$\bullet$}+(-1,-1) node (21) {$\bullet$}+(1,-1) node (22) {$\bullet$}+(3,-1) node (23) {$\bullet$}+(5,-1) node (23b) {$\bullet$}+(2,1) node (0) {$\bullet$}; 
\draw[->] (0) edge node[left] {$c$} (12);
\draw[->] (21) edge node[above] {$a$} (11);
\draw[->] (22) edge node[above] {$b$} (11);
\draw[->] (22) edge node[above] {$a$} (12);
\draw[->] (23) edge node[above] {$b$} (12);
\draw[->] (23) edge node[above] {$a$} (13);

\draw[->] (23b) edge node[above] {$b$} (13);
\end{tikzpicture}
\\\vs
\begin{tikzpicture}[scale=1]
\draw (0,0) node (11) {$\bullet$} +(2,0) node (12) {$\bullet$}+(4,0) node (13)  {$2$}+(-1,-1) node (21) {$\bullet$}+(1,-1) node (22) {$\bullet$}+(3,-1) node (23) {$\bullet$}+(5,-1) node (23b) {$\bullet$}+(4,1) node (0) {$\bullet$}; 
\draw[->] (0) edge node[left] {$c$} (13);
\draw[->] (21) edge node[above] {$a$} (11);
\draw[->] (22) edge node[above] {$b$} (11);
\draw[->] (22) edge node[above] {$a$} (12);
\draw[->] (23) edge node[above] {$b$} (12);
\draw[->] (23) edge node[above] {$a$} (13);

\draw[->] (23b) edge node[above] {$b$} (13);
\end{tikzpicture}
\\\vs
\begin{tikzpicture}[scale=1]

\draw (0,0) node (11) {$\bullet$} +(2,0) node (12) {$\bullet$}+(4,0) node (13)  {$\bullet$}+(6,0) node (14) {$\bullet$}+(-1,-1) node (21) {$\bullet$}+(1,-1) node (22) {$\bullet$}+(3,-1) node (23) {$\bullet$}+(5,-1) node (23b) {$\bullet$}+(6,1) node (0) {$\bullet$}; 
\draw[->] (0) edge node[left] {$c$} (14);
\draw[->] (21) edge node[above] {$a$} (11);
\draw[->] (22) edge node[above] {$b$} (11);
\draw[->] (22) edge node[above] {$a$} (12);
\draw[->] (23) edge node[above] {$b$} (12);
\draw[->] (23) edge node[above] {$a$} (13);
\draw[->] (23b) edge node[above] {$b$} (13);
\draw[->] (23b) edge node[above] {$a$} (14);
\end{tikzpicture}
\end{center}
 Note that the dimension vectors are exceptional roots and thus there exists a unique representation attached to each. Moreover, the weights of the weight spaces can be obtained easily.  
For $n=1$ there exists a $\mathbb P^1$-family of indecomposable representations, all of which are stable. The groups $U_{\psi_{T_{1}^1}}$ and  $U_{\psi_{T_2^1}}$ are trivial and thus $(1,1,1)$ admits a cellular tree normal form given by the lifted attracting sets
\begin{equation}\att(T^1_1)=((\ast),(1),(1)),\qquad\att(T^1_2)=((1),(0),(1)).\end{equation}

For $n=2$, the attracting sets can be computed as
\begin{equation}\att(T^2_1)/U_{\psi_{T^2_1}}\cong\left(\begin{pmatrix}0&\ast\\ 1&\ast\end{pmatrix},\begin{pmatrix}1&0\\ 0&1\end{pmatrix},\begin{pmatrix}1\\0\end{pmatrix}\right),\,\att(T^2_2)/U_{\psi_{T^2_2}}\cong\left(\begin{pmatrix}1&0\\ 0&\ast\end{pmatrix},\begin{pmatrix}0&0\\0&1\end{pmatrix},\begin{pmatrix}1\\1\end{pmatrix}\right),\end{equation}\begin{equation}\att(T^2_3)/U_{\psi_{T^2_3}}\cong\left(\begin{pmatrix}1&0\\ 0&1\end{pmatrix},\begin{pmatrix}0&1\\ 0&0\end{pmatrix},\begin{pmatrix}0\\1\end{pmatrix}\right).\end{equation}
Thus we again get a cellular tree normal form for the stable representations. This is also true for general $n$ as the following recursive construction shows.

Write $E_n$ for the identity matrix of size $n$ and $J_k(0)$ for the Jordan block of size $k$ with eigenvalue $0$. Moreover, write
$C_k^n=\att(T_k^n)/U_{\psi_{T^n_k}}.$ For $n\geq 2$, we have
\begin{equation}\att(T_1^n)/U_{\psi_{T^n_1}}=\left(J_{n}(0)^T+\begin{pmatrix}&\,&\,&\ast\\&\,\text{\Huge0}&\,&\vdots\\&\,&\,&\ast\end{pmatrix}, E_{n},e_1\right),\end{equation}
\begin{equation}\att(T_k^n)/U_{\psi_{T^2_1}}=\left(\begin{pmatrix}1&0&\hdots &0\\0&&&\\\vdots&&(C^{n-1}_{k-1})_a&\\0&&&\end{pmatrix}, \begin{pmatrix}J_{k-1}(0)&0\\0&E_{n-k+1}\end{pmatrix},e_{k-1}+e_{k}\right)\,\text{for } k=2,\ldots,n+1\end{equation}
where we set $e_{n+1}:=0$. In total, we obtain
\begin{lemma}
The set of stable representations $R^{\Theta-\mathrm{st}}_{(n,n,1)}(K(2,1))$ admits a cellular tree normal form. Furthermore, the moduli space $M^{\Theta-\mathrm{st}}_{(n,n,1)}(K(2,1))$ admits a cell decomposition 
 into affine spaces with $n+1$ cells in total. Since there exists precisely one cell of each dimension, the Poincar\'{e} 
  polynomial (in singular cohomology) is given by
\begin{equation*}P^{\Theta-\mathrm{st}}_{(n,n,1)}(q)=\sum_{i=0}^n q^{2i}.\end{equation*}
\end{lemma}

\section{Applications, examples and discussion}\label{sec:applications}
\noindent We give several examples of roots for which our methods can be used to classify all indecomposables up to isomorphism by constructing a cellular tree normal form. It seems that our methods are particularly useful to classify Schurian representations of a fixed root. Thus the next step could be to consider those roots, for which every indecomposable representation is already Schurian.  

\subsection{Complexity of classification: Harder-Narasimhan length and Schur level}Let $\Theta\in\Z^{Q_0}$ be a linear form defining a stability and let $M\in\rep(Q)$. Recall that $M\in\rep(Q)$ admits a unique subrepresentation $\scss(M)$ which is of maximal dimension under those subrepresentations with maximal slope, see \cite[Section 4]{Reineke:2008fk}. These subrepresentations can successively be used to build the so-called Harder-Narasimhan filtration
\begin{equation}0=M_0\subset M_1\subset\ldots\subset M_s=M\end{equation}
with $\Theta$-semistable subquotients $M_i/M_{i-1}$ satisfying $\mu(M_i/M_{i-1})>\mu(M_{i+1}/M_i)$ for $i=0,\ldots,s-1$, which is also unique.

During this section, we frequently use the following lemma.

\begin{lemma}\label{scss} Let $M\in\rep(Q)$. Then we have $\Hom(\scss{M},M/\scss{M})=0$.
\end{lemma}
\begin{proof}If $M$ and $N$ are two semistable representations such that $\mu(M)>\mu(N)$, then we have $\Hom(M,N)=0$, see \cite[Lemma 4.2]{Reineke:2008fk}. Assume that
\begin{equation}0=M_0\subset M_1\subset \ldots\subset M_{s-1}\subset M_s=M\end{equation}
 is the HN-filtration of $M$.

Then we have $M_1=\scss(M)$. Moreover, the subquotients $M_i/M_{i-1}$ are semistable such that $\mu(M_i/M_{i-1})>\mu(M_{i+1}/M_i)$ for $i=1,\ldots, s-1$. Consider the induced chain of epimorphisms
\begin{equation}M/M_1\xrightarrow{\pi_1} M/M_2\xrightarrow{\pi_2}\ldots\xrightarrow{\pi_{s-2}}M/M_{s-1}\end{equation} 
with $\ker(\pi_i)=M_{i+1}/M_{i}$. Assume that $0\neq f\in\Hom(M_1,M/M_1)$. 

Following the introductory remark, we have $\Hom(M_1,\ker(\pi_i))=0$ which inductively yields - using the universal property of the kernel -  that $\pi_i\circ\ldots\circ\pi_1\circ f\neq 0$ and thus a nonzero homomorphism $\pi_{s-2}\circ\ldots\circ\pi_1\circ f:M_1\to M/M_{s-1}$. As $M/M_{s-1}$ is semistable with $\mu(M_1)>\mu(M/M_{s-1})$, this yields a contradiction.
\end{proof}

\begin{definition} Let $\Theta$ be a linear form and $\alpha\in\ZZ_{\geq 0}^{Q_0}$ be a root. Let $M\in\rep(Q)$ with HN-filtration as above.
\begin{enumerate}\item  We write $\mathrm{hn}^{\Theta}(M)=s$ for the length of the Harder-Narasimhan filtration of $M$ and call it HN-length of $M$ in the following.
\item We define the HN-length of $\alpha$ by
\begin{equation*} \mathrm{hn}^{\Theta}(\alpha):=\max\{\mathrm{hn}^{\Theta}(M)\mid M\in\mathrm{Rep}(Q)\text{ indecomposable},\,\udim M=\alpha\}.\end{equation*}
\item We define the Schur level of $\alpha$ by 
 \begin{equation*}\scl(\alpha):=\max\{\dim\End(M)\mid M\in\mathrm{Rep}(Q)\text{ indecomposable},\,\udim M=\alpha\}.\end{equation*}
\end{enumerate}
\end{definition}
We often suppress $\Theta$ in $\hn^{\Theta}$ if it is fixed. The following examples suggest that - besides the Euler form $\langle\alpha,\alpha\rangle$ - both $\hn(\alpha)$ and $\scl(\alpha)$ give a measure for the complexity of the classification problem for indecomposables of dimension vector $\alpha$. For instance, if $\alpha$ is $\Theta$-coprime and if we have $\hn(\alpha)=1$, every representation is stable. Thus all indecomposables are parametrized by a smooth projective variety and we also have $\scl(\alpha)=1$. If, moreover, there exists a torus action with finitely many torus fixed points, the moduli space of stables admits a cell decomposition inducing a tree normal form for $\alpha$. Note that this is clearly true for exceptional representations. Another example is the following.

\begin{example}We extend Example \ref{ex:grassmannian} and consider the dimension vector $(1,d)$ of $K(m)$. For $\Theta=(1,0)$, we have $\hn((1,d))=\scl((1,d))=1$. Actually, the description of the indecomposables in Example \ref{ex:grassmannian} yields that every indecomposable representation is stable with respect to this stability. Thus we have 
\begin{equation}M_{(1,d)}^{(1,0)-\mathrm{st}}(K(m))\cong \Gr_{d}(\kk^m).\end{equation}
On this moduli space we can choose a torus action as defined in Section \ref{sec:torusaction} such that the torus fixed points are precisely the ${m\choose d}$ tree modules of dimension $(1,d)$. More precisely, these tree modules are obtained when colouring the arrows of the bipartite graph with one source and $d$ sinks in the colours $\{a_1,\ldots, a_m\}$ in such a way that the colours of the arrows are pairwise disjoint. 

The corresponding Bia{\l}ynicki-Birula decomposition of the moduli space with ${m\choose d}$ cells can be identified with the Schubert decomposition of the Grassmannian. This gives a cellular tree normal form for the indecomposable representations of dimension $(1,d)$.
\end{example}

In general, it would be interesting to investigate the following question as it seems that $\hn^\Theta(\alpha)$ limits the possible values for $\scl(\alpha)$ if $\Theta$ defines a nontrivial stability condition.
\begin{question}Is there a connection between $\scl(\alpha)$ and $\hn^\Theta(\alpha)$? 
\end{question}

Also the following lemma suggests that the invariants $\scl(\alpha)$ and $\hn(\alpha)$ measure how difficult the classification problem is.
\begin{lemma}\label{lem:specialroots}Fix a linear form $\Theta$ and let $\alpha$ be a root such that $\mo\neq\emptyset$. Assume furthermore that all roots $\beta<\alpha$ are $\Theta$-coprime. Let $M\in R_\alpha(Q)$ be an unstable Schurian representation with $\hn(M)=2$. Then $\dim\End(\scss{M})=\dim\End(M/\scss{M})=1$, $\sk{\udim\scss{M}}{\udim\scss{M}}<\sk{\alpha}{\alpha}$ and $\sk{\udim M/\scss{M}}{M/\udim\scss{M}}<\sk{\alpha}{\alpha}$ hold.
\end{lemma}
\begin{proof}
Write $U:=\scss{M}$, $V=M/\scss{M}$ and consider the short exact sequence $\ses{U}{M}{V}.$ By Lemma \ref{scss}, we have $\Hom(U,V)=0$. As $M$ is Schurian, it follows that $\Hom(V,M)=\Hom(M,U)=0$. As $\udim U$ and $\udim V$ are $\Theta$-coprime by assumption, stability and semi-stability coincide for the dimension vectors $\udim U$ and $\udim V$. Thus $\hn(\alpha)=2$ gives that $U$ and $V$ are Schurian. Applying the respective $\Hom$-functors, these observations can be used to obtain the following diagram:

\begin{equation}\label{bigD}\xymatrix{&\Hom(U,U)=\kk\ar@{^{(}->}[d]&\Hom(U,M)=\kk\ar^0[d]&0\ar[d]&\\
\Hom(V,V)=\kk\ar@{^{(}->}[r]&\Ext(V,U)\ar[d]\ar[r]&\Ext(V,M)\ar[d]\ar[r]&\Ext(V,V)\ar[d]\ar[r]&0\\
\Hom(M,V)=\kk\ar^0[r]&\Ext(M,U)\ar[d]\ar[r]&\Ext(M,M)\ar[d]\ar[r]&\Ext(M,V)\ar[d]\ar[r]&0\\
0\ar[r]&\Ext(U,U)\ar[d]\ar[r]&\Ext(U,M)\ar[d]\ar[r]&\Ext(U,V)\ar[d]\ar[r]&0\\
&0&0&0}\end{equation}

In particular, we have $\dim\Ext(U,U)\leq\dim\Ext(M,M)$ and $\dim\Ext(V,V)\leq\dim\Ext(M,M)$. If $\dim\Ext(U,U)=\dim\Ext(M,M)$, we have $\Ext(U,V)=0$. As $\dim\Ext$ is upper-semicontinuous, \cite[Theorem 3.3]{sc92} then shows that a general and thus every   representation of dimension $\udim M$ has a subrepresentation of dimension $\udim U$. But then $\mu(U)\geq \mu(M)$ contradicts the existence of stable representations. 

Thus we have $\dim\Ext(U,U)<\dim\Ext(M,M)$ and by duality $\dim \Ext(V,V)<\dim \Ext(M,M)$. We get
\begin{equation}\sk{U}{U}=\dim\Hom(U,U)-\dim\Ext(U,U)=1-\dim\Ext(U,U)<1-\dim\Ext(M,M)=\sk{\alpha}{\alpha}\end{equation}
and the same for $\sk{V}{V}$.
\end{proof}

\begin{remark}\label{rem:hn}If $\alpha$ is a root as in Lemma \ref{lem:specialroots} with $\scl(\alpha)=1$ and $\hn(\alpha)\leq 2$, then every indecomposable representation of dimension $\alpha$ is either stable or it can be written as a middle term of a short exact of stable representations $U$ and $V$ with $\Hom(U,V)=0$.

The other way around, every pair of stable representations $(U,V)$ with $\mu(U)>\mu(V)$, i.e. $\Hom(U,V)=0$, and $\udim U+\udim V=\alpha$ can be used to construct indecomposable representations of dimension $\alpha$ using the methods of Section \ref{sec:homological}. Thus the classification problem of indecomposables - which is a purely algebraically term - translates into a geometric problem of classifying stable representation of dimension $\alpha$ and of dimension $\beta<\alpha$. In Section \ref{sec:isotropic}, we apply this to isotropic Schur roots $\delta$ with $\scl(\delta)=1$. In Section \ref{sec:extsub}, we give another example of a class of roots where this lemma applies. In these cases, also a cellular tree normal form is derived.

Actually, for these kind of roots, it seems likely that we always get a cellular tree normal form with our methods rather directly even if a proof is missing. For general roots $\alpha$, we can at least construct mosaics of indecomposable representations with them. In any case, our investigations should help organizing indecomposable representations of a fixed root $\alpha$. 

\end{remark}


\subsection{Isotropic Schur roots}\label{sec:isotropic}
We consider isotropic Schur roots $\delta$ which, moreover, satisfy $\scl(\delta)=1$ and construct a cellular tree normal form for $\delta$. For a root $\alpha$, define $\Theta_\alpha\in\Hom(\Z^{Q_0},\Z)$ by
\begin{equation}\Theta_\alpha=\sk{-}{\alpha}-\sk{\alpha}{-}.\end{equation}
This linear form defines a stability condition in the sense of King \cite[Definition 1.1]{Kin94}. Recall that this stability condition is equivalent to a stability condition in our sense when defining a linear form $\hat\Theta_\alpha$ by $\hat\Theta_\alpha=\mu\cdot\dim-\Theta_\alpha$ for arbitrary $\mu\in\Z$.

Also recall that a representation $M$ is (semi)-stable with respect to $\Theta_\alpha$ if $\Theta_\alpha(M)=0$ and if $\Theta_\alpha(U)>0$ (resp. $\Theta_\alpha(U)\geq 0$) for all proper subrepresentations $U\subset M$.

\begin{lemma}\label{lem:isotropic} Let $\delta$ be an isotropic Schur root such that $\scl(\delta)=1$ and let $M$ be an indecomposable representation of dimension $\delta$.
\begin{enumerate}\item The representation $M$ is semistable with respect to $\Theta_\delta$. In particular, we have $\hn(\delta)=1$.
\item  The representation $M$ is unstable if and only if there exist a nonsplit exact sequence $\ses{U}{M}{V}$ where $U$ and $V$ are two exceptional representations with $\Hom(U,V)=\Hom(V,U)=0$ and $\Ext(U,V)=\Ext(V,U)=\kk$.
\end{enumerate}
\end{lemma} 
\begin{proof} Let $M$ be an indecomposable representation of dimension $\delta$ and $U=\scss{M}$. Write $V=M/U$. If $M$ were not semistable, using \eqref{eq:Euler}, we had $U\neq M$ with \begin{equation}0>\Theta_\alpha(U)=\dim\Hom(U,M)-\dim\Ext(U,M)-\dim\Hom(M,U)+\dim\Ext(M,U).\end{equation}

As $\End(M)=\kk$ by assumption, we have $\Hom(M,U)=0$. Moreover, we have $\dim\Ext(U,M)\leq \dim \Ext(M,M)=1$. As $\Hom(U,V)=0$ by Lemma \ref{scss}, it follows that $1\leq\dim\Hom(U,U)=\dim\Hom(U,M)$. Thus the inequality cannot hold and it follows that $U=M$.

Moreover, equality can only hold if $\dim\Hom(U,M)=\dim\Ext(U,M)=1$ and $\dim\Ext(M,U)=0$. Then we may consider a diagram as the one obtained in \eqref{bigD}. It follows that $\Ext(U,U)=0$, $\Ext(U,V)=\kk$ and thus $\Ext(M,V)=\kk$ which yields $\Ext(V,V)=0$. We also get $\Ext(V,U)=\kk$. Thus $U,V$ is a pair of exceptional representations satisfying the claimed conditions. 

The other way around this setup of $U$ and $V$ yields an indecomposable representation $M$ of dimension $\delta$ as the middle term of any nonsplit sequence $\ses{U}{M}{V}$, see Lemma \ref{lem:sesmorphism}, which is unique as $\Ext(V,U)=\kk$. As
$\Theta_{\alpha}(U)=0$, the representation is not stable. But it is semistable by the first part of the lemma.
\end{proof}

For an exceptional root $\beta$, we write $M_\beta$ for the unique indecomposable representation (up to isomorphism) of dimension $\beta$. For two dimension vectors $\alpha$ and $\beta$, recall that the maps $\dim\Hom(-,-):R_\alpha(Q)\times R_\beta(Q)\to\Z^{\geq 0}$ and $\dim\Ext(-,-):R_\alpha(Q)\times R_\beta(Q)\to\Z^{\geq 0}$ are upper-semicontinuous, see e.g. \cite{sc92}. Its minimal - and thus general - value is denoted by $\hom(\alpha,\beta)$ and $\ext(\alpha,\beta)$ respectively.
\begin{proposition}\label{pro:isotropic}

Let $\delta$ be an isotropic Schur with $\scl(\delta)=1$. Then there exist two exceptional roots $\beta_1$ and $\beta_2$ with $\dim\Ext(M_{\beta_2},M_{\beta_1})=2$, $M_{\beta_2}\in M_{\beta_1}^{\perp}$ and $\Hom(M_{\beta_2},M_{\beta_1})=0$. Moreover, every stable representation $M$ of dimension $\delta$ can be written as the middle term of a short exact sequence of the form
\begin{equation*}\ses{M_{\beta_1}}{M}{M_{\beta_2}}.\end{equation*} 
\end{proposition}
\begin{proof}
The first part is \cite[Proposition 4.1]{PW16}. 

Thus assume that $M$ is stable of dimension $\delta$. As $\ext(\beta_1,\beta_2)=0$, by \cite[Theorem 3.3]{sc92}, every representation of dimension $\delta$ has a subrepresentation of dimension $\beta_1$. So we may assume that $W\subset M$ is a subrepresentation with $\udim W=\beta_1$. Furthermore, assume that $W\ncong M_{\beta_1}$ which already means that $W$ is decomposable. Since $M$ is stable, for every direct summand $W_i$ of $W$, we have $\Theta_\delta(W_i)>0$. Since we have \begin{equation}\Theta_\delta(W)=\Theta_\delta(\beta_1)=\sk{\beta_1}{\delta}-\sk{\delta}{\beta_1}=1-(-1)=2,\end{equation} we conclude that $W\cong W_1\oplus W_2$ for indecomposable representations $W_i$ such that $-\sk{M}{W_1}=\sk{W_2}{M}=1$. 

As $\End(M)=\kk$, we have $\Hom(M,W_i)=0$ for $i=1,2$. Thus we can deduce that $\Ext(M,W_1)=\kk$ and $\Ext(M,W_2)=0$. This already shows that $W_2$ is exceptional by applying $\Hom(-,W_2)$ to $W_2\subset M$. Moreover, it shows that $\Ext(W_1,W_2)=0$ by applying $\Hom(-,W_2)$ to $W_1\subset M$. Writing $\udim W_i=\gamma_i$, this yields
\begin{align}\label{ineq1}\sk{\beta_1}{\gamma_2}=\sk{\udim W}{\udim W_2}=\sk{\udim W_1+\udim W_2}{\udim W_2}\geq 1.\end{align}

On the other hand, as $\Hom(M,W_2)=\Ext(M,W_2)=0$, we have $\delta\in{^\perp}\gamma_2$.  By \cite[Theorem 4.1]{sc92}, it follows that $\ext(\gamma_2,\delta)=0$ or $\hom(\gamma_2,\delta)=0$. Now the considerations from above show that $\ext(\gamma_2,\delta)=0$ and thus $\hom(\gamma_2,\delta)=1$. As $\ext(\delta,\gamma_2)=0$ and $\gamma_2<\delta$, the Happel-Ringel lemma \cite[Lemma 4.1]{HR82} implies that a general representation of dimension $\delta$ has a subrepresentation of dimension $\gamma_2$. 

Let $M_\delta$ be a general representation of dimension $\delta$. As $\hom(\beta_1,\delta)=1$ and $\gamma_2<\beta_1$, it follows that $\hom(\beta_1,\gamma_2)=0$. Indeed, otherwise there would be a non-injective morphism $M_{\beta_1}\to M_\delta$ factoring through $W_2$ which contradicts the fact that the unique (up to scalars) nonzero morphism $M_{\beta_1}\to M_\delta$ is injective.

Since a general representation of dimension $\delta$ has a subrepresentation of dimension $\beta_1$, using $\ext(\delta,\gamma_2)=0$, we get $\ext(\beta_1,\gamma_2)=0$. In summary, this yields $\sk{\beta_1}{\gamma_2}=0$ which yields a contradiction to inequality (\ref{ineq1}). In particular, such a subrepresentation $W$ cannot exist and every subrepresentation of dimension $\beta_1$ is isomorphic to $M_{\beta_1}$.

By duality, every quotient of dimension $\beta_2$ of a stable representation of dimension $\delta$ is forced to be isomorphic to $M_{\beta_2}$. As every stable representation has a subrepresentation of dimension $\beta_1$ and thus a quotient of dimension $\beta_2$, the claim follows.
\end{proof}

\begin{remark} The considerations of \cite{PW16} give a way to determine the desired decomposition of an isotropic root into a pair of exceptional roots.

Actually, the first part also seems to follow inductively by the algorithm of Derksen and Weyman \cite{DW02}, see also \cite[Proposition 3.15]{Wei12} for possible decompositions of isotropic roots.

\end{remark}
\begin{theorem}\label{thm:isotropic}
Let $\delta$ be an isotropic Schur root with $\scl(\delta)=1$. Then $\delta$ admits a cellular tree normal form.
\end{theorem}
\begin{proof}
Proposition \ref{pro:isotropic} together with Theorem \ref{thm:cells} shows that there exists a $\mathbb P^1$-family of nonisomorphic indecomposables which can be written as middle terms of short exact sequences
\begin{equation}\ses{M_{\beta_1}}{M}{M_{\beta_2}}\end{equation}
which gives a mosaic $\{(M_i,U_i)\}_{i=1,2}$ of indecomposables of dimension $\delta$. As $M_{\beta_1}$ and $M_{\beta_2}$ are exceptional, they are tree modules by \cite{Ringel:1998gf}. Thus we can actually apply Theorem \ref{thm:treecells} when choosing a tree-shaped basis of $\Ext(M_{\beta_2},M_{\beta_1})$ which is compatible with the coefficient quivers. It follows that we can choose $M_i$ as tree modules implying that every representation $M_i(\lambda)$ with $\lambda\in U_i$ has a $(M_i,U_i)$-normal form.

As all stable representations are covered by this, it remains to consider the unstable (but semistable) indecomposable representations. They are covered by Lemma \ref{lem:isotropic} and the same argument shows that they are actually tree modules.
\end{proof}

\begin{remark}
For isotropic Schur roots $\delta$ with $\scl(\delta)\geq 2$, the proofs show that there exists a mosaic of indecomposable representations which gives a cellular tree normal form for Schurian representations of dimension $\delta$. As these representations form a dense subset of all indecomposables, it remains to investigate the finitely many non-Schurian indecomposables in this case. 

\end{remark}

\begin{remark} If $\delta$ is an isotropic Schur root with $\scl(\delta)=1$, each Schurian representation $M_1$ of dimension $\delta$ gives rise to an indecomposable representation $M_n$ of dimension $n\delta$ for $n\geq 1$. They can be successively found as the middle terms of the unique nonsplit short exact sequences
\begin{equation}\ses{M_1}{M_n}{M_{n-1}}.\end{equation}
So we can use the cellular tree normal form for $\delta$ to construct mosaics of indecomposables of dimension $n\delta$. Note that we inductively get
$\Hom(M_n,M_1)\cong\Hom(M_{n-1},M_1)=\kk$, $\Hom(M_1,M_1)\cong\Ext(M_{n-1},M_1)=\kk$ and $\Ext(M_n,M_1)\cong\Ext(M_1,M_1)=\kk$.

In the case of extended Dynkin quivers, all indecomposables of dimension $n\delta$ can be constructed in this way. For general multiples of isotropic Schur roots, this does not seem to follow from the general theory.

\end{remark}
\begin{example}\label{ex:subspace n4}
The case of the dimension vector $(1,1)$ of $K(2)$ was treated in Example \ref{ex:kronecker}. 

Let us consider the imaginary Schur root $\delta=(2,1,1,1,1)$ of $S(4)$ as defined in Section \ref{sec:subspace}. Then $\delta$ decomposes into $\delta=\beta_1+\beta_2=(1,1,0,0,0)+(1,0,1,1,1)$. The representations $M_{\beta_1}$ and $M_{\beta_2}$ are given by the coefficient quivers
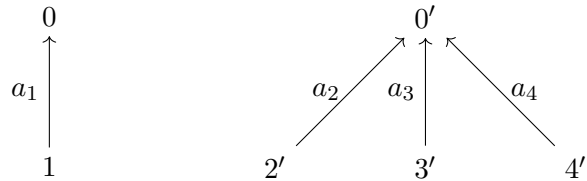
\begin{center}
\begin{tikzpicture}

\draw (0,0) node (J1) {$0$} +(0,-2) node (I1) {$1$}+(5,0) node (J2) {$0'$} +(3,-2) node (I2) {$2'$} +(5,-2) node (I3) {$3'$}+(7,-2) node (I4) {$4'$}; 
\draw[->] (I1)--(J1) node[left,pos=.5] {$a_1$};
\draw[->] (I2)--(J2) node[left,pos=.5] {$a_2$};
\draw[->] (I3)--(J2) node[left,pos=.5] {$a_3$};
\draw[->] (I4)--(J2) node[right,pos=.5] {$a_4$};
\end{tikzpicture}
\end{center}
A tree-shaped basis of $\Ext(M_{\beta_2},M_{\beta_1})$ is represented by
$\{2'\xrightarrow{a_2}0,3'\xrightarrow{a_3}0\}.$ According to Theorem \ref{thm:treecells}, we obtain a mosaic of indecomposables consisting of a one- and a zero-dimensional cell
\begin{equation}\left\{\left((\kk^2,\kk,\kk,\kk,\kk),(\begin{pmatrix}1\\0\end{pmatrix},\begin{pmatrix}1\\1\end{pmatrix},\begin{pmatrix}\ast\\1\end{pmatrix},\begin{pmatrix}0\\1\end{pmatrix})\right),\left((\kk^2,\kk,\kk,\kk,\kk),(\begin{pmatrix}1\\0\end{pmatrix},\begin{pmatrix}0\\1\end{pmatrix},\begin{pmatrix}1\\1\end{pmatrix},\begin{pmatrix}0\\1\end{pmatrix})\right)\right\}.\end{equation} 

The remaining indecomposables can be found as described in the second part of Lemma \ref{lem:isotropic} when considering the following decompositions of $\delta$ into exceptional roots:
\begin{equation}\delta=(1,1,1,0,0)+(1,0,0,1,1)=(1,1,0,1,0)+(1,0,1,0,1)=(1,1,0,0,1)+(1,0,1,1,0).\end{equation}
Note that, we only get three new indecomposables as the other three indecomposables are covered by the first mosaic. More precisely, the remaining indecomposables are given by the three tree modules defined by the coefficient quivers
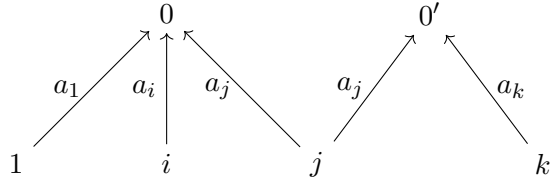
\begin{center}
\begin{tikzpicture}

\draw (0,0) node (J1) {$0$} +(-2,-2) node (I1) {$1$}+(3.5,0) node (J2) {$0'$} +(0,-2) node (I2) {$i$} +(2,-2) node (I3) {$j$}+(5,-2) node (I4) {$k$}; 
\draw[->] (I1)--(J1) node[left,pos=.5] {$a_1$};
\draw[->] (I2)--(J1) node[left,pos=.5] {$a_i$};
\draw[->] (I3)--(J1) node[left,pos=.5] {$a_j$};
\draw[->] (I3)--(J2) node[left,pos=.5] {$a_j$};
\draw[->] (I4)--(J2) node[right,pos=.5] {$a_k$};

\end{tikzpicture}
\end{center}
with $\{i,j,k\}=\{2,3,4\}$.
                 
\end{example}

\begin{example}If we add a vertex $q_{5}$ and an arrow $a_5:q_5\to q_1$ to $S(4)$, the isotropic Schur root
\begin{equation}\delta=3q_0+2q_1+2q_2+q_3+q_4+q_5\end{equation}
satisfies $\scl(\delta)\geq 2$. In this case 
\begin{equation}\delta=(2q_0+q_1+q_2+q_3+q_5)+(q_0+q_1+q_2+q_4)\end{equation}
is a decomposition into exceptional roots as in Proposition \ref{pro:isotropic} giving a $\Pn^1$-family of non-isomorphic Schurian indecomposables. But there also exists a decomposition
\begin{equation}\delta=2\cdot \beta_1+\beta_2+\beta_3=2\cdot (q_0+q_1+q_2+0+0)+(q_0+q_3+q_4)+(q_5).\end{equation}
As we have $\hom(\beta_i,\beta_j)=0$ for $i\neq j$, \cite[Theorem 3.3]{Wei15} shows that every indecomposable of dimension $(1,2,1)$ of the quiver
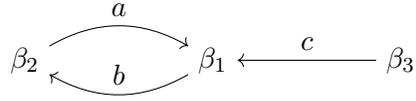
\begin{center}
\begin{tikzpicture}[scale=1]

\draw (0,0) node (I) {$\beta_2$} +(2.5,0) node (J) {$\beta_1$}+(5,0) node (K) {$\beta_3$}; 
\draw[->] (I) edge [bend left=30] node[above] {$a$} (J);
\draw[->] (J) edge [bend left=30] node[above] {$b$} (I);

\draw[->] (K) edge node[above] {$c$} (J);
\end{tikzpicture}
\end{center}                                     
gives an indecomposable of dimension $\delta$ with the same endomorphism ring. As the tree module defined by the coefficient quiver
\begin{equation}\beta_1^1\xrightarrow{b}\beta_2\xrightarrow{a}\beta_1^2\xleftarrow{c}\beta_3\end{equation}
has a two-dimensional endomorphism ring, we indeed get $\scl(\delta)\geq 2$.

In this case, it is still doable, but more difficult, to classify all indecomposable representations. But, this example seems to be a good starting point to analyze isotropic Schur roots with $\scl(\alpha)\geq 2$ in general. Actually, the methods of the paper should also be applicable in these cases. 
\end{example}

\subsection{Extended subspace quiver: an example}\label{sec:extsub}
We consider the quiver 
\begin{equation}\label{eq:Tnquiver}
T(n)=(\{q_0,q_1,\ldots,q_{n+1}\},\{a_1,a_2:q_1\to q_0\}\cup\{b_i:q_{i+1}\to q_0\mid i=1,\ldots,n\})
\end{equation}
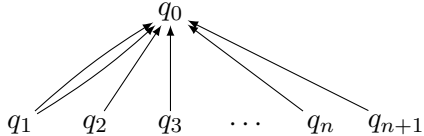
\[
\vcenter{\hbox{\begin{tikzpicture}[point/.style={shape=circle,fill=black,scale=.5pt,outer sep=3pt},>=latex]
  \node[outer sep=-2pt] (q0) at (0,1.5) {${q_0}$};
   \node[outer sep=-2pt] (q1) at (-2,0) {${q_1}$};
   \node[outer sep=-2pt] (q2) at (-1,0) {${q_2}$};
  \node[outer sep=-2pt] (q3) at (0,0) {${q_3}$};
   \node[outer sep=-2pt] (d) at (1,0) {${\cdots}$};
   \node[outer sep=-2pt] (qn) at (2,0) {${q_n}$};
  \node[outer sep=-2pt] (qn1) at (3,0) {${q_{n+1}}$};
  \path[->]
        (q1) edge[bend left=5]  (q0)
				(q1) edge[bend right=5]  (q0)
				
        (q2) edge (q0)
        (q3) edge (q0)
        (qn) edge (q0)
        (qn1) edge  (q0);
        	   \end{tikzpicture}}}
\]
and the root $\alpha(n)=nq_0+\sum_{i=1}^{n+1}q_i$. In this case, we can classify the indecomposables as described in Remark \ref{rem:hn}. We choose the stability defined by the linear form $\Theta=(0,1,\ldots,1)$ so that the stable and semistable points of dimension $\alpha(n)$ coincide. Let $M$ be a representation of dimension $\alpha(n)$. Each $I\subset\{1,\ldots, n+1\}$ naturally defines a subrepresentation $M_I$ of $M$. Define $d(M)_I=\dim (M_I)_0$. A representation $M$ is stable if and only if
\begin{equation}d(M)_I>\frac{n}{n+1}|I|\text{ for all }\emptyset\neq I\subsetneq\{1,\ldots,n+1\}\end{equation}
if and only if
\begin{equation}d(M)_I\in\{|I|,|I|+1\}\text{ for all }\emptyset\neq I\subsetneq\{1,\ldots,n+1\}.\end{equation}

We choose a torus action on $M_{\alpha(n)}^{\Theta-\mathrm{st}}(T(n))$ by defining $\gamma=(1,2,0,\ldots,0)$. Then it is straightforward to check that the torus fixed points are given by the exceptional representations of the indicated dimension which are supported at the following subquiver of the universal abelian covering quiver
 
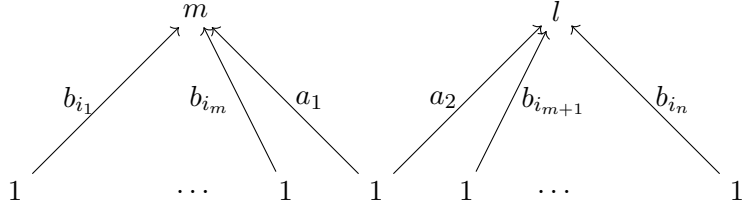
\begin{center}
\begin{tikzpicture}[scale=0.8]

\draw (0,0) node (I1) {$m$} +(6,0) node (I2) {$l$}+(-3,-3) node (J1) {$1$}+(0,-3) node (J2) {$\dots$}+(1.5,-3) node (J3) {$1$}  +(3,-3) node (J4) {$1$}+(4.5,-3) node (J5) {$1$}+(6,-3) node (J6) {$\dots$} +(9,-3) node (J7) {$1$}; 

\draw[->] (J1)--(I1) node[pos=.5,left] {$b_{i_1}$};
\draw[->] (J3)--(I1) node[pos=.5,left] {$b_{i_m}$};
\draw[->] (J4)--(I1) node[pos=.5,right] {$a_1$};
\draw[->] (J4)--(I2) node[pos=.5,left] {$a_2$};
\draw[->] (J5)--(I2) node[pos=.5,right] {$b_{i_{m+1}}$};
\draw[->] (J7)--(I2) node[pos=.5,right] {$b_{i_{n}}$};
\end{tikzpicture}
\end{center}
Here $\{i_1,\ldots,i_{n}\}=\{1,\ldots,n\}$ and $m+l=n$. As all torus fixed points are exceptional representations of $\widehat{T(n)}$, there exist precisely ${n\choose m}$ fixed points of this kind which we denote by $T(i_1,\ldots,i_m)$. In the next step, we can apply Theorem \ref{thm:geomcells} and find strong and separating subspaces $U_{T(i_1,\ldots,i_m)}\subset R(T(i_1,\ldots,i_m),T(i_1,\ldots,i_m))$ of dimension $m$ which are induced by the respective attracting cells. For $i_j=j$, they are given as follows and in general by the obvious modification:
\begin{equation}T(1,\ldots,m)+U_{T(1,\ldots,m)}=\left(\sum_{i=1}^me_i,\begin{pmatrix}\ast\\\vdots\\\ast\\1\\\vdots\\1\end{pmatrix},e_1\ldots,e_m,e_{m+1},\ldots,e_n\right).\end{equation}
Note that we use this to determine the Poincaré polynomial: as there are ${n\choose m}$ cells of dimension $m$, we obtain
\begin{equation}P_{\alpha(n)}^{\Theta-\mathrm{st}}(q)=\sum_{m=0}^n{n\choose m}q^{2m}=(1+q^2)^n.\end{equation}
To classify all indecomposables, it remains to investigate the unstable indecomposable representations. The next two lemmas are to describe these kind of representations.

\begin{lemma}\label{lem:finalexample1} Let $M$ be a representation of $T(n)$ with $\udim M=\alpha(n)$. If $M$ is indecomposable, but unstable, there exists $m\geq 1$ and $I\subset\{1,\ldots, n\}$ with $|I|=m+1$ and $1\notin I$ such that $\udim\scss M=mq_0+\sum_{i\in I}q_i$. In particular, $\scss M$ is exceptional. Finally, we have that $M/\mathrm{scss} M$ is indecomposable. 
\end{lemma}
\begin{proof}
If $M$ is unstable, there exists $\emptyset\neq I\subsetneq\{1,\ldots,n\}$ such that $d(M)_I<|I|$. Let $\hat I:=\{1,\ldots,n+1\}\backslash I$. If we had $d(M)_I\leq |I|-2$, from $d(M)_{\hat I}\leq |\hat I|+1$, it can easily be seen that $M$ were decomposable. Indeed, we either have $M_0\cong (M_I)_0\oplus (M_{\hat I})_0$ inducing a direct sum decomposition of $M$ or even $M_I\cap M_{\hat I}\neq \{0\}$ which makes the simple representation $S_{q_0}$ a direct summand. Thus it follows that $d(M)_I=|I|-1$.

If $1\in I$ with $d(M)_I<|I|$, we had $d(M)_{\hat I}\leq |\hat I|$ and thus the same argument shows that $M$ were decomposable. Thus $M$ has a subrepresentation of the form as claimed. Now it can be shown inductively that $\mathrm{scss}(M)$ is the subrepresentation of this form such that $m$ is minimal.

Write $U=\mathrm{scss}(M)$. We have $\udim M/U=(n-m)q_0+\sum_{i\in \hat I}q_i$ with $|\hat I|=n-m$ and $1\in\hat I$. If $M/U$ were decomposable, it is straightforward that $M/U$ had a direct summand $V$ of dimension $lq_0+\sum_{i\in I'}q_i$ with $|I'|=l<n-m$ and $1\notin I'$. But then we had $\Ext(V,U)=0$ contradicting the indecomposability of $M$.
\end{proof}
We continue proceeding along Remark \ref{rem:hn} and classify the possible quotients $M/\scss M$.

\begin{lemma}\label{lem:finalexample2}
Let $\beta(n)=nq_0+\sum_{i=1}^n q_i\in\ZZ_{\geq 0}^{T(n)_0}$. Then the indecomposables of dimension $\beta(n)$ can be parametrized by $\mathbb P^1$. Furthermore, $\beta(n)$ admits a cellular tree normal form.
\end{lemma}
\begin{proof}First note that, applying the BGP-reflection \cite{bgp} functor to the sink $q_0$, the dimension vector $\beta(n)$ becomes one at $q_0$. For the reflected dimension vector, Example \ref{ex:extkronecker} for the dimension vector $(1,1,1)$ can be generalized in such a way showing that the indecomposables are parametrized by $\Pn^1$. 

To obtain a cellular tree normal form for $\beta(n)$ itself, it is now convenient to apply Theorem \ref{thm:geomcells} because together with the previous observation it shows that all representations are stable with respect to a certain stability. Actually, there exists a torus action on the moduli space with two torus fixed points $T_1$ and $T_2$ inducing the following cells of indecomposables
\begin{equation}U_1=\att(T_1)=\left(\begin{pmatrix}1\\1\\\vdots\\1\\\ast\end{pmatrix},e_n,e_1,\ldots,e_{n-1}\right),\,U_2=\att(T_2)=\left(e_n,\begin{pmatrix}1\\1\\\vdots\\1\\ 0\end{pmatrix},e_1,\ldots,e_{n-1}\right).\end{equation}
This gives a mosaic $\{(T_1,T_1-U_1),(T_2,\{0\})\}$ parametrizing all indecomposables and thus a cellular tree normal form for $\beta(n)$.
\end{proof}

\begin{theorem}
The dimension vectors $\alpha(n)$ admit a cellular tree normal form.
\end{theorem}
\begin{proof}
By Lemma \ref{lem:finalexample1} it follows that, for every unstable indecomposable $M$, there exists a short exact sequence of the form
\begin{equation}\ses{\mathrm{scss}(M)}{M}{M/\mathrm{scss}(M)}\end{equation}
with stable kernel and indecomposable quotient. Moreover, we have $\gamma(I):=\udim \mathrm{scss}(M)=mq_0+\sum_{i\in I}q_i$, $|I|=m+1$ and $1\notin I$ for some $1\leq m\leq n$. In particular, $\mathrm{scss}(M)$ is exceptional and thus a tree module. Furthermore, $M/\mathrm{scss}(M)$ is indecomposable of dimension $\beta(\hat I):=|\hat I|q_0+\sum_{i\in \hat I}q_i$ with $|\hat I|=n-m$.

The other way around, first note that 
$\dim\Ext(N,\mathrm{scss}(M))=m$ for every indecomposable $N$ with $\udim N=\beta(\hat I)$. In particular, for each such $I$ we can apply Theorem \ref{thm:treecells} and Lemma \ref{lem:finalexample2} to the pairs of cells $(M_{\gamma(I)},\{0\}),(T^I_1,T^I_1-U^I_1))$ and $(M_{\gamma(I)},\{0\}),(T^I_2,\{0\}))$ to obtain a $\Pn^{m-1}\times(T_1^I- U_1^I)$- and a $\Pn^{m-1}$-family of indecomposables respectively. Here $M_{\gamma(I)}$ is the exceptional of dimension $\gamma(I)$ and $(T^I_1,T_1^I-U^I_1)$ is the obvious modification of $(T_1,T_1-U_1)$ constructed in Lemma \ref{lem:finalexample2}.

As all representations in the cells from above have a tree normal form, every unstable indecomposable is obtained in this way and as all constructed indecomposables are nonisomorphic, this shows that $\alpha(n)$ admits a cellular tree normal form.
\end{proof}


\bibliographystyle{alpha}
\bibliography{KinserWeist}

\end{document}